\documentclass{birkjour_t2}
 \theoremstyle{definition}
 
 \theoremstyle{remark}

 \numberwithin{equation}{section}

\usepackage{graphicx}
\usepackage{amssymb, amsmath, amsthm}
\usepackage{amsfonts}
\usepackage{amssymb}
\usepackage{float}
\usepackage{mathrsfs}
\usepackage{enumitem}
\newtheorem{theorem}{Theorem}

\newtheorem{lemma}[theorem]{Lemma}

\newtheorem{remark}[theorem]{Remark}

\newcommand{\N}{\mathbb N}
\newcommand{\Z}{\mathbb Z}

\newcommand{\R}{\mathbb R}

\newcommand{\ds}{\displaystyle}

\renewcommand{\theequation}{\thesection.\arabic{equation}}
\numberwithin{equation}{section}
\numberwithin{theorem}{section}

\numberwithin{figure}{section}

\begin{document}

%
%
%
%
%
%
%
%
%

\title[Magnetogeostrophic equations]
 {Existence, uniqueness, regularity and instability results for the viscous magneto-geostrophic equation}

\author{Susan Friedlander}

\address{Department of Mathematics\\
University of Southern California}

\email{susanfri@usc.edu}

\author{Anthony Suen}

\address{Department of Mathematics and Information Technology\\
The Hong Kong Institute of Education}

\email{acksuen@ied.edu.hk}

\date{}



\begin{abstract}
We study the three dimensional active scalar equation called the magneto-geostropic equation which was proposed by Moffatt and Loper as a model for the geodynamo processes in the Earth's fluid core. When the viscosity of the fluid is positive, the constitutive law that relates the drift velocity $u(x,t)$ and the scalar temperature $\theta(x,t)$ produces two orders of smoothing. We study the implications of this property. For example, we prove that in the case of the non-diffusive ($\varepsilon_\kappa=0$) active scalar equation, initial data $\theta_0\in L^3$ implies the existence of unique, global weak solutions. If $\theta_0\in W^{s,3}$ with $s>0$, then the solution $\theta(x,t)\in W^{s,3}$ for all time. In the case of positive diffusivity ($\varepsilon_\kappa>0$), even for singular initial data $\theta_0\in L^3$, the global solution is instantaneously $C^\infty$-smoothed and satisfies the drift-diffusion equation classically for all $t>0$. We demonstrate, via a particular example, that the viscous magneto-geostrophic equation permits exponentially growing ``dynamo type" instabilities.
\end{abstract}

\maketitle
\section{Introduction}

Physicists have long realized the importance of the Earth's magnetic field and that this field originates in
the Earth's fluid core. The geodynamo is the process by which the rotating, convecting, electrically conducting
molten iron in the Earth's fluid core maintains the geomagnetic field against ohmic decay. The
convective processes in the core that produce the velocity fields required for dynamo action are a combination
of thermal and compositional convection. The full dynamo problem requires the examination of the full
three dimensional partial differential equations governing convective, incompressible magnetohydrodyamics
(MHD). In the past decades computer models have been used to simulate the actual geodynamo. However, current computers and numerical methods require the imposition of diffusivities that are several orders of
magnitude larger than those which are realistic. It is therefore reasonable to attempt to gain some insight into
the geodynamo by considering a reduction of the full MHD equations to a system that is more tractable, but
one that retains many of the essential features of the problem. The magnetogeostrophic equation proposed
by Moffatt and Loper (\cite{ML94}-\cite{M08}) is one such model. The physical postulates of this model are the
following: slow cooling of the Earth leads to slow solidification of the liquid metal core onto the solid
inner core and releases latent heat of solidification that drives compositional convection in the fluid core.
The arguments for the appropriate ranges of the characteristic length, velocity, and perturbation density are
based on these physical postulates.

We first present the full coupled three-dimensional MHD equations for the evolution of the velocity vector $u(x,t)$, the magnetic field vector $B(x,t)$ and the buoyancy (or temperature) field $\theta(x,t)$ in the Boussinesq approximation
and written in the frame of reference rotating with angular velocity $\Omega$. The physical forces governing this
system are Coriolis forces, Lorentz forces, and gravity. Following the notation of Moffatt and Loper \cite{ML94}
we write the equations in terms of dimensionless variables. The orders of magnitude of the resulting nondimensional
parameters are motivated by the physical postulates of the Moffatt and Loper model:

\begin{align}
N^2[R_0(\partial_t u+u\cdot\nabla u)+e_3\times u]&=-\nabla P+e_2\cdot\nabla b+R_m b\cdot\nabla b+N^2\theta e_3+\varepsilon_\nu \Delta u,\label{1.1-1}\\
R_m[\partial_t b+u\cdot\nabla b-b\cdot\nabla u]&=e_2\cdot\nabla u+\Delta b,\label{1.1-2}\\
\partial_t\theta+u\cdot\nabla\theta&=\varepsilon_\kappa\Delta\theta+S,\label{1.1-3}\\
\nabla\cdot u&=0,\nabla\cdot b=0\label{1.1-4}.
\end{align}

Here $S=S(x,t)$ is a given smooth function that represents the forcing of the system. The mathematical statement of the geodynamo problem asks whether there are initial data for the MHD
system for which the evolution of the perturbation of the magnetic field $b(x,t)$ grows for sufficiently long
time. This can be interpreted as a question of the existence of instabilities of \eqref{1.1-1}-\eqref{1.1-4}.

The notation in \eqref{1.1-1}-\eqref{1.1-4} is the following. The Cartesian unit vectors are denoted by $e_1, e_2$ and $e_3$. For simplicity, we have assumed that the axis of rotation and the gravity $g$ are aligned in the direction of
$e_3$. We have assumed that the magnetic field $B(x,t)$ consists of the sum of an underlying constant field
$B_0e_2$ and a perturbation $b(x,t)$. Our choice of $B_0e_2$ as the underlying magnetic field is consistent with the
models where the magnetic field is believed to be predominantly toroidal (see \cite{ML94}). The dimensionless parameters in \eqref{1.1-1}-\eqref{1.1-4} are $N^2$ (the inverse Elsasser number), $R_0$ (the Rossby number), $R_m$ (the magnetic Reynolds number), $\varepsilon_\nu$ (the inverse square of the Hartman number) and $\varepsilon_\kappa$ (the inverse Peclet number). The definitions of these numbers in terms of the relevant dimensionless quantities are given in \cite{ML94} where the authors argue that in the region of the fluid core that they are modeling the parameters have the following orders of magnitude: $N^2$ is order unity while $R_0$ is order $10^{-3}$ and $R_m$ is relatively small. The Moffatt and Loper model neglects the terms multiplied by $R_0$ and $R_m$ in comparison with the remaining terms. Essentially this means that the evolution equations for the velocity and magnetic field take a simplified ``quasi-static'' form and are linear in the perturbation vector fields $u(x,t)$ and $b(x,t)$. The diffusion parameters $\varepsilon_\nu$, which is proportional to the viscosity $\nu$, and $\varepsilon_\kappa$, which is proportional to the thermal diffusivity $\kappa$, are plausibly extremely small. However their roles are to multiply the Laplacian which is the highest spatial derivatives in the system and these terms are retained in the model.

In the Moffatt-Loper model the dominant balance of the leading order terms in equations \eqref{1.1-1}-\eqref{1.1-2} gives the following reduced system

\begin{align}
N^2 e_3\times u&=-\nabla P+e_2\cdot\nabla b+N^2\theta e_3+\varepsilon_\nu\Delta u,\label{1.1-5}\\
0&=e_2\cdot\nabla u+\Delta b,\label{1.1-6}\\
\nabla\cdot u&=0, \nabla\cdot b=0.\label{1.1-7}
\end{align}
The linear system \eqref{1.1-5}-\eqref{1.1-7} determines the differential operators that relate the vector fields $u(x,t)$ and $b(x,t)$ with the scalar buoyancy (or temperature) $\theta(x,t)$. These operators encode the vestiges of the physics in the problem, namely Coriolis force, Lorentz force, and gravity. Vector manipulations of \eqref{1.1-5}-\eqref{1.1-7} give the expression
\begin{align}\label{1.1-8}
\left\{[\varepsilon_\nu\Delta^2-(e_2\cdot\nabla)^2]^2+N^4(e_3\cdot\nabla)^2\Delta\right\}u=-N^2[\varepsilon_\nu\Delta^2-(e_2\cdot\nabla)^2]&\nabla\times(e_3\times\nabla\theta)\notag\\
\qquad&+N^4(e_3\cdot\nabla)\Delta(e_3\times\nabla\theta).
\end{align}
The sole remaining nonlinearity in the system comes from the coupling of \eqref{1.1-8} with the equation \eqref{1.1-3} for the time evolution of $\theta(x,t)$. We call this nonlinear active scalar equation the {\it magnetogeostrophic} (MG)
equation.

The physically relevant boundary for a model of the Earth's fluid core is a spherical annulus. However for the purposes of studying the mathematical properties of the MG equations we consider the more tractable case of the active scalar equation in the domain $\mathbb{T}^3\times(0,\infty)=[0,2\pi]^3\times(0,\infty)$ (i.e. with periodic boundary conditions). Without loss of generality we may assume that $\int_{\mathbb{T}^3} \theta (x, t) dx = 0$ for all $t \ge 0$, since the mean of $\theta$ is conserved by the flows. We study the active scalar equation
\begin{align}
\label{1.1} \left\{ \begin{array}{l}
\partial_t\theta+u\cdot\nabla\theta=\varepsilon_\kappa\Delta\theta, \\
u=M[\theta],\theta(x,0)=\theta_0(x)
\end{array}\right.
\end{align}
via an examination of the Fourier multiplier operator $M$ obtained from \eqref{1.1-8}, which relates $u$ and $\theta$. More precisely,
\begin{align*}
u_j=M_j [\theta]=(\widehat{M_j}\hat\theta)^\vee
\end{align*}
for $j\in\{1,2,3\}$ and the explicit expression for the components of $\widehat M$ as functions of the Fourier variable $k=(k_1,k_2,k_3)\in\Z^3$ with $k_3\neq0$ are
\begin{align}
\widehat M_1(k)&=[N^4k_2k_3|k|^2-N^2k_1k_3(k_2^2+\varepsilon_\nu|k|^4)]D(k)^{-1},\label{1.1-10}\\
\widehat M_2(k)&=[-N^4k_1k_3|k|^2-N^2k_2k_3(k_2^2+\varepsilon_\nu|k|^4)]D(k)^{-1}\label{1.1-11},\\
\widehat M_3(k)&=[N^2(k_1^2+k_2^2)(k_2^2+\varepsilon_\nu|k|^4)]D(k)^{-1},\label{1.1-12}
\end{align}
where
\begin{align}\label{1.1-13}
D(k)=N^4|k|^2k_3^2+(\varepsilon_\nu|k|^4+k_2^2)^2.
\end{align}
On the set $\{k_3=0\}$, we let $M_j(k)=0$ since for self-consistency of the model we assume that $\theta$ and $u$ both have zero vertical mean. We point out that $k_j\cdot \widehat M_j(k) = 0$ and hence the velocity field $u$ given by \eqref{1.1} is indeed divergence-free.

In a series of recent papers (\cite{FRV12}, \cite{FV11a}, \cite{FV11b}, \cite{FV12}) properties were proved for the {\it inviscid} MG equation (i.e. the system \eqref{1.1}-\eqref{1.1-13} when the viscosity $\varepsilon_\nu$ is set to zero). In this case the Fourier multiplier symbols $\widehat M_{\varepsilon_\nu=0}$, given by \eqref{1.1-10}-\eqref{1.1-12} with $\varepsilon_\nu=0$ are not bounded in all regions of Fourier space. More specifically in ``curved'' regions where $k_3=\mathcal{O}(1)$, $k_2=\mathcal{O}(|k_1|^\frac{1}{2})$ the symbols are unbounded as $|k_1|\rightarrow\infty$ with $|\widehat M_{\varepsilon_\nu=0}(k)|\le C|k|$ for some positive constant $C$. Thus the relation between the velocity field $u$ and the scalar $\theta$ is given by a {\it singular} operator of order 1. The implications of this fact for the inviscid MG equation are summarized in the survey paper Friedlander-Rusin-Vicol \cite{FRV14}. In particular, it is proved that when the thermal diffusivity $\varepsilon_\kappa$ is set to zero, the inviscid MG equation is {\it ill-posed} in the sense of Hadamard in Sobolev spaces. In contrast when $\varepsilon_\kappa>0$, the inviscid but thermally dissipative MG equation is globally well-posed.

In this present article we study the {\it viscous} MG equation (i.e. $\varepsilon_\nu>0$). The situation is then dramatically different because the operator $M$ whose symbols are given by \eqref{1.1-10}-\eqref{1.1-12} now is bounded on $k\in \Z^3 \backslash \{k_3 = 0\}$ and $\lim\limits_{|k| \to \infty} \widehat{M}_j \to \lim\limits_{|k| \to \infty} \frac{1}{\varepsilon_\nu |k|^2}$. Thus the constitutive law for the {\it viscous} MG model produces two orders of smoothing.

We remark that the MG equations fall into a hierarchy of active scalar equations arising in fluid dynamics in terms of the nature of the operator that produces the drift velocity from the scalar field:
\begin{enumerate}
\item for the inviscid MG equation the operator is singular of order 1.
\item for the surface quasi-geostrophic equation the operator is the Riesz transform which is singular of order zero.
\item for the 2D Euler equation in vorticity form the operator is smoothing of degree one.
\item for the viscous MG equation the operator is smoothing of degree two.
\end{enumerate}
In this sense the viscous MG equation, even without thermal diffusion, is ``better behaved'' than the 2D Euler equation. In the following sections of this article we will prove strong properties of the system \eqref{1.1}-\eqref{1.1-13} that are a consequence of this fact.

The paper is organized as follows. In Section~2, we state the main results of this article which are summarized in various theorems. In Section~3, we introduce some notations and recall some useful embeddings which can be found in the literature. In Section 4, we study the case when $\varepsilon_\kappa=0$. We prove that initial data $\theta_0\in L^3 $ implies the existence of unique, global weak solutions; while for $\theta_0\in W^{1,3} $ we obtain the single exponential growth in time on $\|\nabla\theta(\cdot,t)\|_{L^3 }$. In Section~5, we consider the thermally dissipative case ($\varepsilon_\kappa>0$) and prove that the solution $\theta^{\varepsilon_\kappa}(x,t)\in L^3 $ exists for all time. In particular, we show that the solution is instantaneously $C^\infty$-smoothed out and in the class $W^{s,p} $ for all positive time with $s\in[0,1)$ and $p\in (3,\infty)$. In Section~6, we address the convergence of $\theta^{\varepsilon_\kappa}$ as $\varepsilon_\kappa$ approaches to zero. Under the assumption that $\theta_0\in L^3 $, the sequence of solutions $\{\theta^{\varepsilon_\kappa}(x,t)\}_{\varepsilon_\kappa>0}$ to \eqref{1.1}-\eqref{1.1-13} with $\varepsilon_\kappa>0$ converges  to $\theta(x,t)$ weakly in $L^3 $ for all $t>0$, where $\theta(x,t)$ is the solution to \eqref{1.1}-\eqref{1.1-13} when $\varepsilon_\kappa=0$. Moreover, if $\nabla\theta_0\in L^2$, then for all $T>0$, the term $\displaystyle\varepsilon_\kappa\int_0^T\!\!\!\int|\nabla\theta^{\varepsilon_\kappa}|^2$ tends to zero as $\varepsilon_\kappa\rightarrow0$. In Section 7, we illustrate some dynamical features of the {\it forced} MG equation by demonstrating that the Moffalt and Loper model can sustain very rapid exponentially growing instabilities. We examine perturbations of a background temperature gradient and we construct unstable eigenvalues in the spectrum of the linearized equation. We use a suitable modification of the method of continued fractions first introduced in the context of the Navier-Stokes equations by Meshalkin and Sinai \cite{MS61}. The anisotropy of the MG symbols $\widehat{M}_j$ given by (1.10)-(1.13) permits eigenvalues that are very large when the parameters $\varepsilon_\nu$ and $\varepsilon_\kappa$ are extremely small.

\section{Main results}

We now give a precise formulation of our results. The following theorems will be proved in Section~4 to Section~6.

\begin{theorem} Assume that $\theta_0\in L^3$ has mean zero on $\mathbb{T}^3$ and $\varepsilon_\kappa=0$. There exists unique global weak solution to \eqref{1.1}-\eqref{1.1-13} such that
\begin{align}
\theta&\in BC((0,\infty);L^3),\label{1.3}\\
u&\in C((0,\infty);W^{2,3}).\label{1.4}
\end{align}
In particular, $\theta(\cdot,t)\rightarrow\theta_0$ weakly in $L^3$ as $t\rightarrow0^+$. Here $BC$ stands for {\it ``bounded continuous functions''}.
\end{theorem}


\begin{theorem} Assume that $\theta_0\in W^{s,3}$ has mean zero on $\mathbb{T}^3$ with $s>0$ and $\varepsilon_\kappa=0$. There exists a unique solution to \eqref{1.1}-\eqref{1.1-13} such that $\theta(\cdot,t)\in W^{s,3}$ for all $t\ge0$ with $\theta(\cdot,t)\rightarrow\theta_0$ weakly in $L^3$ as $t\rightarrow0^+$. In particular, for $s=1$, we have the following single exponential growth in time on $\|\nabla\theta(\cdot,t)\|_{L^3}$:
\begin{align}
\|\nabla\theta(\cdot,t)\|_{L^3}\le \mathcal{C}_1 \|\nabla\theta_0\|_{L^3}\exp\left(t\,\mathcal{C}_2\|\theta_0\|_{W^{1,3}}\right),\label{1.5}
\end{align}
where $\mathcal{C}_1,\mathcal{C}_2>0$ are constants which depend only on the spatial dimension.
\end{theorem}

\begin{theorem} Assume that $\varepsilon_\kappa>0$ in \eqref{1.1}-\eqref{1.1-13}. Let $\theta_0\in L^3$ be given and has mean zero on $\mathbb{T}^3$. Then there exists a unique global-in-time mild solution $\theta^{\varepsilon_\kappa}$ to \eqref{1.1}-\eqref{1.1-13} such that
\begin{align}
\theta^{\varepsilon_\kappa}&\in BC((0,\infty);L^3)\label{1.4-1},\\
t^{\frac{s}{2}+\frac{1}{2}-\frac{3}{2p}}\theta^{\varepsilon_\kappa}&\in C((0,\infty);\dot{W}^{s,p})\label{1.4-2},
\end{align}
for all $s,p$ satisfying $s\in[0,1)$ and $p\in(3,\infty)$. In particular, $\theta^{\varepsilon_\kappa}(\cdot,t)\rightarrow\theta_0$ in $L^3$ as $t\rightarrow0^+$ and $\|\theta^{\varepsilon_\kappa}(\cdot,t)\|_{\dot{W}^{s,p}}\rightarrow0$ as $t\rightarrow\infty$.
\end{theorem}

\begin{remark} If we assume $\theta_0\in \dot{W}^{\tilde{s},3}$ for $\tilde{s}>0$, then by a similar argument given in the proof of Theorem~2.3, one can show that there exists a unique global-in-time mild solution $\theta^{\varepsilon_\kappa}$ to \eqref{1.1}-\eqref{1.1-13} with
\begin{align*}
\theta^{\varepsilon_\kappa}&\in BC((0,\infty);L^3 ),\\
t^{\frac{-(s-\tilde{s})}{2}+\frac{1}{2}-\frac{3}{2p}}\theta^{\varepsilon_\kappa}&\in C((0,\infty);\dot{W}^{s,p} ),
\end{align*}
for all $s,p$ satisfying $s\in[\tilde{s},\tilde{s}+1)$ and $p\in(3,\infty)$, and $\|\theta^{\varepsilon_\kappa}(\cdot,t)\|_{\dot{W}^{s,p} }\rightarrow0$ as $t\rightarrow\infty$.
\end{remark}

\begin{theorem}
Let $\theta_0\in L^3 $ be given and has mean zero on $\mathbb{T}^3$. Let $\theta^{\varepsilon_\kappa}$ be the solution to \eqref{1.1}-\eqref{1.1-13} when $\varepsilon_\kappa>0$ with initial data $\theta_0$ as obtained in Theorem~2.3. There exists a sequence $\{\varepsilon_{\kappa_n}\}_{n\in\N}$ with $\displaystyle\lim_{n\rightarrow\infty}\varepsilon_{\kappa_n}=0$ such that,
\begin{align}\label{1.6}
\mbox{$\theta^{\varepsilon_{\kappa_n}}(\cdot,t)\rightarrow\theta(\cdot,t)$ weakly in $L^3 $ as $n\rightarrow\infty$, for every $t\ge0$,}
\end{align}
where $\theta$ is the solution to \eqref{1.1}-\eqref{1.1-13} when $\varepsilon_\kappa=0$ with initial data $\theta_0$ as obtained in Theorem~2.1.

Moreover, if we further assume that $\nabla\theta_0\in L^2 $, then for any $T>0$, we have
\begin{align}\label{1.7}
\lim_{\varepsilon_\kappa\rightarrow0}\varepsilon_\kappa\int_0^T\!\!\!\int|\nabla\theta^{\varepsilon_\kappa}|^2dxds=0.
\end{align}
\end{theorem}

\section{Preliminaries}

We introduce the following notations. We say $(\theta,u)$ is a weak solution to \eqref{1.1}-\eqref{1.1-13} for $\varepsilon_\kappa=0$ if they solve the system in the weak sense, that means for all $\phi\in C^{\infty}_0(\mathbb{T}^3\times(0,\infty),\R^3)$, we have
\begin{align*}
\int_0^\infty\int_{\mathbb{T}^3}(\partial_t \phi+u\cdot\nabla\phi)\theta(x,t)dxdt+\int_{\mathbb{T}^3}\phi(x,0)\theta_0(x)dx=0.
\end{align*}
$W^{s,p}$ and $\dot{W}^{s,p}$ are the usual inhomogeneous Sobolev space and homogeneous Sobolev space with norm $\|\cdot\|_{W^{s,p} }$ and $\|\cdot\|_{\dot{W}^{s,p} }$ respectively. We also define $\|\cdot\|_{L.L. }$ to be the Log-Lipschitz norm given by
\begin{align*}
\|f\|_{L.L. }=\sup_{x\neq y}\frac{|f(x)-f(y)|}{|x-y||(1+|\log|x-y||)}.
\end{align*}

We recall the following facts from the literature (see for example Azzam-Bedrossian \cite{AB11}, Bahouri-Chemin-Danchin \cite{BCD11} and Ziemer \cite{Z89}): there exists a constant $C>0$ such that

\begin{align}
\|f\|_{L.L. }&\le C \|\nabla f\|_{BMO },\label{2.2}\\
\|f\|_{BMO }&\le C \|f\|_{W^{1,3} },\label{2.1}\\
\|f\|_{L^\infty }&\le C \|f\|_{W^{2,3} },\label{2.3}\\
\|f\|_{L^6 }&\le C\|\nabla f\|_{L^2 }, \label{2.3-1}
\end{align}
and for $p\ge1$ and $q>3$, there are constants $C(p),C(q)>0$ such that
\begin{align}
\|f\|_{L^\infty }&\le C(q) \|f\|_{W^{1,q} }.\label{2.6}
\end{align}
For simplicity, we write $\|\cdot\|_{L^p}=\|\cdot\|_{L^p(\mathbb{T}^3)}$, $\|\cdot\|_{W^{s,p}}=\|\cdot\|_{W^{s,p}(\mathbb{T}^3)}$, etc. unless otherwise specified.

\section{Non-diffusive case when $\varepsilon_\kappa=0$}

We study the non-diffusive case when $\varepsilon_\kappa=0$ in \eqref{1.1}-\eqref{1.1-13}. We first prove the global-in-time well-posedness of \eqref{1.1}-\eqref{1.1-13} in the Lebesgue space $L^3 $ without any smallness conditions. Here $L^3 $ is the {\it critical} Lebesgue space with respect to the natural scaling of the system \eqref{1.1}-\eqref{1.1-13} in the sense that if $\theta(x,t)$ is a solution, then $\theta_\lambda(x,t)=\lambda^3\theta(\lambda x,\lambda^2 t)$ is also a solution with corresponding drift given by $u_\lambda(x,t)=\lambda u(\lambda x,\lambda^2 t)=M[\theta_\lambda]$ for $\lambda>0$.

We let $(\theta,u)$ be a local smooth solution to \eqref{1.1}-\eqref{1.1-13} with smooth initial data defined on $[0,T]$ for some $T>0$. Under the assumption that $\theta_0\in L^3 $, we show the following two lemmas about some {\it a priori} estimates on $(\theta,u)$.

\begin{lemma}
For any $p\ge1$, there exists $C(p)>0$ such that for all $t\in(0,T)$,
\begin{align}\label{3.1}
\|\theta(\cdot,t)\|_{L^p}\le C(p) \|\theta_0\|_{L^p},
\end{align}
and
\begin{align}\label{3.2}
\|u(\cdot,t)\|_{W^{2,p}}\le C(p) \|\theta_0\|_{L^p}.
\end{align}
\end{lemma}

\begin{proof}
The assertion \eqref{3.1} follows immediately from the first equation in \eqref{1.1} together with the divergence-free property on $u$. To prove \eqref{3.2}, we observe that, by the definition of the operator $M$ given in \eqref{1.1-10}-\eqref{1.1-12},  for all $k\in\mathbb{Z}^3\setminus\{k_3=0\}$ and all $j\in\{1,2,3\}$, we have
\begin{align*}
|\widehat M_j(k)|\le C_{*}|k|^{-2},
\end{align*}
where $C_{*}=C_{*}(N,\varepsilon_\nu)>0$ is a fixed constant. In other words, $M$ is a smoothing operator of degree 2. Hence with the help of the Fourier multiplier theorem (see Stein \cite{S70}), given $p\ge1$, there exists some constant $C(p)>0$ such that
\begin{align*}
\|u(\cdot,t)\|_{W^{2,p}}\le C(p) \|\theta(\cdot,t)\|_{L^p}.
\end{align*}
Therefore \eqref{3.2} follows from \eqref{3.1}.
\end{proof}

\begin{lemma}
There exists a constant $C>0$ such that, for any $t\in(0,T)$, we have
\begin{align}\label{3.3}
\|u(\cdot,t)\|_{L.L.}\le C \|\theta_0\|_{L^3},
\end{align}
and
\begin{align}\label{3.4}
\|u(\cdot,t)\|_{L^\infty}\le C \|\theta_0\|_{L^3}.
\end{align}

\end{lemma}

\begin{proof}
Using \eqref{2.1} and \eqref{2.2}, we can choose some $C>0$ such that
\begin{align*}
\|u(\cdot,t)\|_{L.L.}&\le C \|\nabla u(\cdot,t)\|_{BMO}\\
&\le C \|\nabla u(\cdot,t)\|_{W^{1,3}}\\
&\le C \|u(\cdot,t)\|_{W^{2,3}}.
\end{align*}
On the other hand, with the help of \eqref{3.2} from Lemma~4.1, we can bound $\|u(\cdot,t)\|_{W^{2,3}}$ in terms of $\|\theta_0\|_{L^3}$. Hence \eqref{3.3} follows. \eqref{3.4} can proved similarly by using \eqref{2.3} and \eqref{3.2} and we omit the details.
\end{proof}
\begin{remark} In view of Lemma~4.2, one can obtain bounds on both $\|u(\cdot,t)\|_{L.L. }$ and $\|u(\cdot,t)\|_{L^\infty }$ in terms of $\|\theta_0\|_{L^3 }$ without any further assumption on $\|\theta_0\|_{L^\infty }$. A uniform-in-time bound on the Log-Lipschitzian norm of $u$ is essential to assure the existence and uniqueness of the flow map $\psi(x,t)$ (which will be given in the proof of Theorem~2.1, see below) and hence the existence and uniqueness of the solution. In the case for 2D Euler equation, initial conditions on the vorticity $\omega_0$ of the type $\omega_0\in L^\infty$ (or $\theta_0\in BMO, B^0_{\infty,\infty}$) are required in getting a uniform-in-time bound of $\|u(\cdot,t)\|_{L.L. }$. By utilizing the 2-order smoothing effect of $M$, we obtain enough regularity on $u$ which gives the desired bound on $\|u(\cdot,t)\|_{L.L. }$ in terms of $\|\theta_0\|_{L^3 }$.
\end{remark}
We are ready to give the proof of Theorem~2.1. The main idea is to apply Lemma~4.2 which gives uniform bounds on $\|u(\cdot,t)\|_{L.L. }$ and $\|u(\cdot,t)\|_{L^\infty }$ in terms of $\|\theta_0\|_{L^3 }$ only. Once these bounds are established, the existence and uniqueness follow from the similar argument given by Bernicot-Keraani \cite{BK12} for 2D Euler equation.
\begin{proof}[{\bf Proof of Theorem~2.1}]
The proof is divided into two parts.

\noindent{\bf Existence:} Consider the standard mollifier $\rho\in C^\infty_0 $, and we set $\theta_0^{(n)}=\rho_n *\theta_0$ for $n\in\N$ and $\rho_n(x)=n^3\rho(nx)$. For the rest of this section, $C>0$ denotes a generic constant which is independent of $n$ unless otherwise stated.

By standard argument, we can obtain a sequence of global smooth solution $(\theta^{(n)},u^{(n)})$ to \eqref{1.1} with $u^{(n)}=u(\theta^{(n)})$. Define $\psi_n(x,t)$ to be the flow map given by
\begin{align*}
\partial_t\psi_n(x,t)=u^{(n)}(t,\psi_n(x,t)).
\end{align*}
One can show (for example in \cite{BK12}) that
\begin{align}\label{3.5}
\|\psi_n(t,\cdot)\|_{*}\le C \exp\left(\int_0^t\|u^{(n)}(\cdot,\tau)\|_{L.L.}d\tau\right),
\end{align}
where the norm $\|\cdot\|_{*}$ is given by
\begin{align*}
\|\psi\|_{*}=\sup_{x\neq y}\Phi(|\psi(x)-\psi(y)|,|x-y|)
\end{align*}
with
\begin{align*}
\Phi(r,s)=\left\{ \begin{array}{l}
\mbox{$\max\{\frac{1+|\log(s)|}{1+|\log(r)|},\frac{1+|\log(r)|}{1+|\log(s)|}\}$, if $(1-s)(1-r)\ge0$,}\\
\mbox{$(1+|\log(s)|)(1+|\log(r)|)$, if $(1-s)(1-r)\le0$.}
\end{array}\right.
\end{align*}
So using \eqref{3.3} and \eqref{3.5}, we have
\begin{align}\label{3.6}
|\psi_n(t,x_2)-\psi_n(t,x_2)|\le c(t)|x_x-x_1|^{\beta(t)}
\end{align}
for all $(x_1,x_2)\in\R^3\times\R$, where $c(t),\beta(t)$ are some continuous functions which depends on $\|u_0\|_{L^3 }$.
And also, for $t_1,t_2\ge0$, using \eqref{3.4},
\begin{align}\label{3.7}
|\psi_n(x,t_1)-\psi_n(x,t_2)|&\le C |t_2-t_1|(\|u(\cdot,t_1)\|_{L^\infty}+\|u(\cdot,t_2)\|_{L^\infty})\notag\\
&\le C |t_2-t_1|\|u_0\|_{L^3}.
\end{align}
In view of \eqref{3.6}-\eqref{3.7}, the family $\{\psi_n\}_{n\in\N}$ is bounded and equicontinuous on every compact set in $\R^+\times\R^3$. Arzela-Ascoli theorem then impies the existence of a limiting trajectory $\psi(x,t)$. By similar analysis on $\{\psi^-_n\}_{n\in\N}$ (where $\psi^-_n$ is the inverse flow map of $\psi_n$), we can obtain a limit $\phi(x,t)$ to $\psi^-_n$ and that $\phi\circ\psi=\psi\circ\phi=id$. So $\psi(x,t)$ is a Lebesgue measure preserving homeomorphism.

We define $\theta(x,t)=\theta_0(\psi^{-}_t (x))$ and $u=M[\theta]$. It follows that
\begin{align*}
\mbox{$\theta^{(n)}(\cdot,t)\rightarrow\theta(\cdot,t)$ in $L^3 $},
\end{align*}
which implies $u^{(n)}(\cdot,t)\rightarrow u(\cdot,t)$ uniformly, using the fact that
\begin{align*}
\|u^{(n)}(\cdot,t)-u(t)\|_{L^\infty}\le C \|\theta^{(n)}(\cdot,t)-\theta(\cdot,t)\|_{L^3}.
\end{align*}
So the above allows us to pass the limit in the integral equation on $\theta^{(n)}$ and prove that $(\theta,u)$ is a weak solution to \eqref{1.1}. The continuity of $\psi$ and the preservation of Lebesgue measure imply that $t\mapsto\theta(t)$ is continuous with values in $L^3 $, in particular implies $u\in C((0,\infty);W^{2,3} )$. Finally, the assertion that $\theta(\cdot,t)\rightarrow\theta_0$ weakly in $L^3$ as $t\rightarrow0^+$ follows by a similar argument as given by DiPerna-Lions \cite{DL89} and Kato-Ponce \cite{KP86}.

 \noindent{\bf Uniqueness:} 
 We only give a sketch of the proof. Let $T>0$ and suppose that $(\theta^1,u^1)$ and $(\theta^2,u^2)$ solve \eqref{1.1}-\eqref{1.1-13} with $\theta^1(\cdot,0)=\theta^2(\cdot,0)=\theta_0$. Following the similar argument given in \cite{BCD11}, there exists a constant $\mathcal{C}>0$ such that for all $\delta\in(0,1)$ and $k\in\{-1\}\cup\N$, we have
\begin{align*}
\|\Delta_k (\theta^1-\theta^2)\|_{L^\infty}\le \mathcal{C}(k+1)2^{k\delta}(\|u^1\|_{\overline{L.L.}}+\|u^2\|_{\overline{L.L.}})\|(\theta^1-\theta^2)\|_{B^{-\delta}_{\infty,\infty}},
\end{align*}
where $\Delta_k$'s are the usual nonhomogeneous dyadic blocks and $\|\cdot\|_{\overline{L.L.}}=\|\cdot\|_{L^\infty}+\|\cdot\|_{L.L.}$. Define
\begin{align*}
\hat{T}=\sup\left\{t\in[0,T]:\mathcal{C}\int_0^t(\|u^1\|_{\overline{L.L.}}+\|u^2\|_{\overline{L.L.}})(\tau)d\tau\le\frac{1}{2}\right\},
\end{align*}
then by \eqref{3.3}-\eqref{3.4}, $\hat{T}$ is well-defined. For each $t\in[0,\hat{T}]$, we let $\displaystyle\delta_t=\mathcal{C}\int_0^t(\|u^1\|_{\overline{L.L.}}+\|u^2\|_{\overline{L.L.}})(\tau)d\tau$. Using Theorem~3.28 in \cite{BCD11}, we have for all $k\ge-1$,
\begin{align*}
2^{-k\delta_t}\|\Delta_k (\theta^1-\theta^2)(t)\|_{L^\infty}\le \frac{1}{2}\sup_{t'\in[0,t]}\|(\theta^1-\theta^2)\|_{B^{-\delta_t}_{\infty,\infty}}.
\end{align*}
Summing over $k$, we conclude that $\theta^1=\theta^2$ on $[0,\hat{T}]$. By repeating the argument a finite number of times we obtain the uniqueness on the whole interval $[0,T]$. We finish the proof of Theorem~2.1.
 \end{proof}

Theorem~2.2 can then be proved by a similar argument as that of Theorem~2.1 for $\theta_0\in W^{s,3} $ with $s>0$. By controlling the term $\|\nabla u(\theta)(\cdot,t)\|_{L^\infty }$, we further obtain the single exponential growth in time on $\|\nabla\theta(\cdot,t)\|_{L^3 }$ when $\theta_0\in W^{1,3} $. The details are given as follows.
\begin{proof}[{\bf Proof of Theorem~2.2}]
We only need {\it a priori} bounds on $\theta$. For the rest of the proof, $C>0$ denotes a generic constant which is independent of time. Given $s>0$, we apply Fourier transform on the first equation of \eqref{1.1}, multiply it by $\langle\xi\rangle^s$, rearrange the terms and take the inverse Fourier transform to obtain
\begin{align*}
\|\theta(\cdot,t)\|_{W^{s,3}}\le C\left[ \|\theta_0\|_{W^{s,3}}+\int_0^t\|\nabla u(\cdot,\tau)\|_{L^\infty}\|\theta(\cdot,\tau)\|_{W^{s,3}}d\tau\right],
\end{align*}
which implies
\begin{align}\label{3.8}
\|\theta(\cdot,t)\|_{W^{s,3}}\le C \|\theta_0\|_{W^{s,3}}\exp\left(\int_0^t C\|\nabla u(\cdot,\tau)\|_{L^\infty}d\tau\right).
\end{align}
It suffices to estimate $\|\nabla u(\cdot,t)\|_{L^\infty }$. We divide it into two cases.

\noindent{\bf Case 1:} $0<s<1$.
Define $p=\frac{3}{1-s}$. Then $p>3$ and we have the following embedding (see for example Nezzaa-Palatuccia-Valdinocia \cite{NPV11}) that $W^{s,3} \hookrightarrow L^p$ and hence
\begin{align}
\|\theta_0\|_{L^p}\le C \|\theta_0\|_{W^{s,3}}.\label{3.9}
\end{align}
On the other hand, by the similar argument as given in the proof of Lemma~3.2, we get
\begin{align}
\|u(\cdot,t)\|_{W^{2,p}}\le C \|\theta(\cdot,t)\|_{L^p}.\label{3.10}
\end{align}
Therefore, using \eqref{2.6}, \eqref{3.1}, \eqref{3.9} and \eqref{3.10}, we conclude
\begin{align*}
\|\nabla u(\cdot,t)\|_{L^\infty}\le C \|\theta_0\|_{W^{s,3}}.
\end{align*}
\noindent{\bf Case 2:} $s\ge1$.
Using the embeddings $W^{\frac{1}{2},3} \hookrightarrow L^6 $ and $W^{s,3} \hookrightarrow W^{\frac{1}{2},3} $, we follow the similar argument as given in Case 1 to get
\begin{align*}
\|\nabla u(\cdot,t)\|_{L^\infty}&\le C \|\theta(\cdot,t)\|_{L^6}\\
&\le C \|\theta_0\|_{L^6}\\
&\le C \|\theta_0\|_{W^{\frac{1}{2},3}}\le C \|\theta_0\|_{W^{s,3}}.
\end{align*}
We substitute the above estimates on $\|\nabla u(\cdot,t)\|_{L^\infty}$ into \eqref{3.8} and obtain the {\it a priori} required bounds on $\theta$.


Finally, we prove the single exponential growth in time on $\|\nabla\theta(\cdot,t)\|_{L^3 }$ under the assumption that $\theta_0\in W^{1,3} $. We consider the sequence of global smooth solution $(\theta^{(n)},u^{(n)})$ to \eqref{1.1} as given in the proof of Theorem~2.1 with mollified initial data $\theta_0^{(n)}\in C^\infty$. We then differentiate the first equation of \eqref{1.1} with respect to $x$, integrate over space-time and use Gronwall's inequality to get
\begin{align}
\|\nabla\theta^{(n)}(\cdot,t)\|_{L^3}\le \mathcal{C}_{1} \|\nabla\theta^{(n)}_0\|_{L^3}\exp\left(\int_0^t \mathcal{C}_{2}\|\nabla u^{(n)} (\cdot,\tau)\|_{L^\infty}d\tau\right),\label{3.11}
\end{align}
where $\mathcal{C}_{1},\mathcal{C}_{2}>0$ are fixed constant which are independent of $\theta_0,n,t$. On the other hand, we have another fixed constant $\mathcal{C}_3>0$ which depends only on the spatial dimension and is independent of $\theta_0,n,t$ such that
\begin{align*}
\mbox{$\|\nabla u^{(n)}(\cdot,t)\|_{L^\infty}\le \mathcal{C}_3 \|\theta^{(n)}_0\|_{W^{1,3}}\le \mathcal{C}_3 \|\theta_0\|_{W^{1,3}}$ for all $t\ge0$ and $n\in\N$.}
\end{align*}
Hence we conclude from \eqref{3.11} that
\begin{align*}
\|\nabla\theta^{(n)}(\cdot,t)\|_{L^3}\le \mathcal{C}_{1} \|\nabla\theta^{(n)}_0\|_{L^3}\exp\left(\int_0^t\,\mathcal{C}_{2} \mathcal{C}_{3}\|\theta_0\|_{W^{1,3}}d\tau\right).
\end{align*}
By taking $n\rightarrow\infty$, \eqref{1.5} follows immediately. We finish the proof of Theorem~2.2.
\end{proof}

\begin{remark} Our result for the single exponential growth in time on $\|\nabla\theta(\cdot,t)\|_{L^3 }$ is better than the well-known double exponential growth in time of the vorticity gradient for the 2D Euler equation (see, for example, Majda-Bertozzi \cite{MB01}, Denisov \cite{D14} for more discussion).
\end{remark}

\begin{remark} We further consider the non-diffusive system \eqref{1.1}-\eqref{1.1-13} with ``damping'', namely
\begin{align}
\label{1.1"}\left\{ \begin{array}{l}
\partial_t\theta+u\cdot\nabla\theta=-c\theta, \\
u=M[\theta],\theta(x,0)=\theta_0(x)
\end{array}\right.
\end{align}
where $c>0$ is the {\it damping constant} and $M$ is the operator as defined in \eqref{1.1-10}-\eqref{1.1-13}. Using a similar argument to the one given in the proof of Theorem~2.2, we have the following estimate on $\nabla \theta$ provided that $\theta_0\in W^{1,3} $:
\begin{align*}
\mbox{$\|\nabla\theta(\cdot,t)\|_{L^3}\le \mathcal{C}_1 \|\nabla\theta_0\|_{L^3}\exp\left[t\left(\mathcal{C}_2 \|\theta_0\|_{W^{1,3}}-c\right)\right]$ for all $t>0$,}
\end{align*}
where $\mathcal{C}_1,\mathcal{C}_2$ are the constants defined in the proof of Theorem~2.2 which depend only on the spatial dimension. Hence the solution $\theta$ to \eqref{1.1"} remains {\it bounded} in $L^{\infty}((0,\infty);W^{1,3} )$ provided that $c>\mathcal{C}_2 \|\theta_0\|_{W^{1,3}}$.
\end{remark}

\section{Thermally diffusive case when $\varepsilon_\kappa>0$}

Next we study the thermally diffusive case when $\varepsilon_\kappa>0$ in \eqref{1.1}-\eqref{1.1-13}. We also mention a recent work obtained by Ferreira-Lima \cite{FL14-1} which proved a global well-posedness result for a family of dissipative active scalar equations (via a different method) with a smallness condition on the weak norm of a Fourier-Besov-Morrey space that allowed to consider some types of large initial data in $L^p$ and Sobolev spaces.

We first state the following lemmas for which the proofs can be found in Carrillo-Ferreira \cite{CF08} and Lewis \cite{L72}.

\begin{lemma}
Let $G(t)$ be the convolution operator with kernel given in Fourier variables by $\widehat g(\xi,t)=e^{-\varepsilon_\kappa t|\xi|^2}$. Then for $s_1\le s_2$, $s_1,s_2\in\R$ and $1\le p_1\le p_2<\infty$, there is a constant $C>0$ such that
\begin{align}\label{4.1}
\|G(t)f\|_{\dot{W}^{s_2,p_2}}\le Ct^{-\frac{(s_2-s_1)}{2}-\frac{3}{2}(\frac{1}{p_1}-\frac{1}{p_2})}\|f\|_{\dot{W}^{s_1,p_1}}
\end{align}
for all $f\in \dot W^{s_1,p_1}$.
\end{lemma}

\begin{lemma}
Let $\mathbb{X}$ be a Banach space with norm $\|\cdot\|_{\mathbb{X}}$ and $B:\mathbb{X}\times \mathbb{X}\rightarrow \mathbb{X}$ be a continuous bilinear map, that means there exists $K>0$ such that
\begin{align*}
\|B(x_1,x_2)\|_{\mathbb{X}}\le K\|x_1\|_{\mathbb{X}}\|x_2\|_{\mathbb{X}}
\end{align*}
for all $x_1,x_2\in \mathbb{X}$. Given $0<\delta<\frac{1}{4K}$ and $y\in \mathbb{X}$ such that $\|y\|_{\mathbb{X}}\le\delta$, there exists a solution $x\in \mathbb{X}$ for the equation
\begin{align*}
x=y+B(x,x)
\end{align*}
such that $\|x\|_{\mathbb{X}}\le2\delta$. The solution $x$ is unique in the closed ball $\{x\in \mathbb{X}:\|x\|_{\mathbb{X}}\le2\delta\}$. Moreover, the solution depends continuously on $y$ in the following sense: if $\|\tilde{y}\|_{\mathbb{X}}\le\delta$, $\tilde{x}=\tilde{y}+B(\tilde{x},\tilde{x})$ and $\|\tilde{x}\|_{\mathbb{X}}\le2\delta$, then
\begin{align*}
\|x-\tilde{x}\|_{\mathbb{X}}\le\frac{\|y-\tilde{y}\|_{\mathbb{X}}}{1-4K\delta}.
\end{align*}
\end{lemma}

We will prove the following theorem about the local-in-time existence of solutions to \eqref{1.1}-\eqref{1.1-13} with initial data $\theta_0\in L^3 $ when $\varepsilon_\kappa>0$. It is crucial in proving Theorem~2.3. The argument is similar to the one given by Ferreira-Lima \cite{FL14}.

\begin{theorem}
Let $\theta_0\in L^3 $. For any $\varepsilon_\kappa>0$, there exists $T>0$ such that \eqref{1.1}-\eqref{1.1-13} has a unique mild solution $\theta^{\varepsilon_\kappa}$ in the class
\begin{align}
\theta^{\varepsilon_\kappa}&\in C((0,T);L^3 ),\label{4.1-1}\\
t^{\frac{s}{2}+\frac{1}{2}-\frac{3}{2p}}\theta^{\varepsilon_\kappa}&\in C((0,T);\dot{W}^{s,p} )\label{4.1-2},
\end{align}
for all $s,p$ satisfying $s\in[0,1)$ and $p\in(3,\infty)$.
\end{theorem}

\begin{proof}[{\bf proof of Theorem~5.3}]
We convert the system \eqref{1.1}-\eqref{1.1-13} into the integral equation:
\begin{align}\label{4.2}
\theta^{\varepsilon_\kappa}(x,t)=G(t)\theta_0(x)+B(\theta^{\varepsilon_\kappa}(x,t),\theta^{\varepsilon_\kappa}(x,t)),
\end{align}
where $G(t)$ is the convolution operator as defined in Lemma~5.1 and $B(\cdot,\cdot)$ is the bilinear form
\begin{align*}
B(\phi(x,t),\psi(x,t))=-\int_0^t G(t-\tau)\left[\nabla\cdot(u(\phi(x,\tau))\psi(x,\tau))\right]d\tau.
\end{align*}

We will prove that there exists a constant $K>0$ such that for any $T\in(0,1]$,
\begin{align}\label{4.3}
\sup_{0<t<T}t^{\frac{1}{2}-\frac{3}{2p}}&\|B(\phi(\cdot,t),\psi(\cdot,t))\|_{L^p}\notag\\
&\le K\left(\sup_{0<t<T}t^{\frac{1}{2}-\frac{3}{2p}}\|\phi(\cdot,t))\|_{L^p}\right)\left(\sup_{0<t<T}t^{\frac{1}{2}-\frac{3}{2p}}\|\psi(\cdot,t)\|_{L^p}\right).
\end{align}
In view of the operator $M$ with $u=M[\theta]$ as given by \eqref{1.1-10}-\eqref{1.1-12}, for any $s\in\R$ and $p\in(1,\infty)$, there exists $K_1>0$ such that given $f\in W^{s,p}$, we have the following estimates on $\|M[f]\|_{W^{{s+2},p}}$:
\begin{align}
\label{4.2-1}
\|M[f]\|_{W^{{s+2},p}}\le K_1\|f\|_{W^{s,p}},
\end{align}
Using \eqref{2.6}, \eqref{4.1} and \eqref{4.2-1}, there exists $K>0$ such that, given $0<t<T$, we have
\begin{align*}
\|B(\phi(\cdot,t),\psi(\cdot,t)\|_{L^p}&\le\int_0^t\|G(t-\tau)\left[\nabla\cdot(u(\phi(\cdot,\tau))\psi(\cdot,\tau))\right]\|_{L^p}d\tau\\
&\le K\int_0^t(t-\tau)^{-\frac{1}{2}}\|\nabla\cdot(u(\phi(\cdot,\tau))\psi(\cdot,\tau))\|_{\dot{W}^{-1,p}}d\tau\\
&\le K\int_0^t(t-\tau)^{-\frac{1}{2}}\|u(\phi(\cdot,\tau))\psi(\cdot,\tau)\|_{L^p}d\tau\\
&\le K\int_0^t(t-\tau)^{-\frac{1}{2}}\|u(\phi(\cdot,\tau))\|_{L^\infty}\|\psi(\cdot,\tau)\|_{L^p}d\tau\\
&\le K\int_0^t(t-\tau)^{-\frac{1}{2}}\|u(\phi(\cdot,\tau))\|_{W^{2,p}}\|\psi(\cdot,\tau)\|_{L^p}d\tau\\
&\le K\int_0^t(t-\tau)^{-\frac{1}{2}}\|\phi(\cdot,\tau)\|_{L^p}\|\psi(\cdot,\tau)\|_{L^p}d\tau\\
&\le K\left[\int_0^t(t-\tau)^{-\frac{1}{2}}\tau^{-1+\frac{3}{p}}d\tau\right]\\
&\qquad\times\left(\sup_{0<\tau<T}\tau^{\frac{1}{2}-\frac{3}{2p}}\|\phi(\cdot,\tau)\|_{L^p}\right)\left(\sup_{0<\tau<T}\tau^{\frac{1}{2}-\frac{3}{2p}}\|\psi(\cdot,\tau)\|_{L^p}\right)\\
&\le K t^{-\frac{1}{2}+\frac{3}{2p}}T^\frac{3}{2p}\left(\int_0^1(1-z)^{-\frac{1}{2}}z^{-1+\frac{3}{p}}dz\right)\\
&\qquad\times\left(\sup_{0<\tau<T}\tau^{\frac{1}{2}-\frac{3}{2p}}\|\phi(\cdot,\tau)\|_{L^p}\right)\left(\sup_{0<\tau<T}\tau^{\frac{1}{2}-\frac{3}{2p}}\|\psi(\cdot,\tau)\|_{L^p}\right).\\
\end{align*}
By the assumption that $T\le1$ and $p\in(3,\infty)$, it implies
\begin{align*}
\mbox{$T^\frac{3}{2p}\le1$ and $\displaystyle\int_0^1(1-z)^{-\frac{1}{2}}z^{-1+\frac{3}{p}}dz<\infty$},
\end{align*}
and so \eqref{4.3} follows.

For any $T\in(0,1]$, we now define
$$\mathcal{E}_T=\{f\mbox{ measurable: }t^{\frac{1}{2}-\frac{3}{2p}}f\in C((0,T);L^p \}$$
and $\displaystyle\|f\|_{\mathcal{E}_T}=\sup_{0<t<T}t^{\frac{1}{2}-\frac{3}{2p}}\|f\|_{L^p}$. It is clear that $\mathcal{E}_T$ is a Banach space, and by \eqref{4.3}, we have
\begin{align*}
\mbox{$\|B(\phi,\psi)\|_{\mathcal{E}_T}\le K\|\phi\|_{\mathcal{E}_T}\|\psi\|_{\mathcal{E}_T}$ for all $\phi,\psi\in\mathcal{E}_T$.}
\end{align*}
Given $\theta_0\in L^3$, using \eqref{4.1} there exists $C_0>0$ such that,
\begin{align*}
\|G(t)\theta_0\|_{L^p}\le C_0t^{-\frac{1}{2}+\frac{3}{2p}}\|\theta_0\|_{L^3},
\end{align*}
so it implies for every $\theta_0\in L^3\cap W^{1,p}$,
\begin{align*}
\lim_{t\rightarrow0^+}t^{\frac{1}{2}-\frac{3}{2p}}\|G(t)\theta_0\|_{L^p}=0.
\end{align*}
Since $\overline{L^3\cap W^{1,p}}^{\|\cdot\|_{L^3}}=L^3$, we conclude that for $\theta_0\in L^3$,
\begin{align*}
\lim_{t\rightarrow0^+}t^{\frac{1}{2}-\frac{3}{2p}}\|G(t)\theta_0\|_{L^p}=0.
\end{align*}
Hence for any $\delta>0$, there exists $T\in(0,1)$ such that
\begin{align*}
\sup_{0<t<T}t^{\frac{1}{2}-\frac{3}{2p}}\|G(t)\theta_0\|_{L^p}\le\delta.
\end{align*}
We can apply Lemma~4.2 to obtain a unique solution $\theta^{\varepsilon_\kappa}$ to \eqref{4.3} such that
\begin{align}\label{4.4}
\sup_{0<t<T}t^{\frac{1}{2}-\frac{3}{2p}}\|\theta^{\varepsilon_\kappa}(\cdot,t)\|_{L^p}\le2\delta.
\end{align}
Next, we consider the Picard sequence defined by:
\begin{align*}
\theta_1(x,t)&=G(t)\theta_0(x),\\
\theta_{n+1}(x,t)&=\theta_1(x,t)+B(\theta_n(x,t),\theta_n(x,t))\mbox{, for $n\in\N$.}
\end{align*}
We notice that the solution $\theta^{\varepsilon_\kappa}$ as given by \eqref{4.4} can be obtained as the limit in $\mathcal{E}_T$ of $\{\theta_n\}_{n\in\N}$. Moreover,
\begin{align}\label{4.5}
\mbox{$\displaystyle\sup_{0<t<T}t^{\frac{1}{2}-\frac{3}{2p}}\|\theta_n(\cdot,t)\|_{L^p}=\|\theta_n\|_{\mathcal{E}_T}\le2\delta$ for all $n\in\N$.}
\end{align}

We now prove that for any $n\in\N$, the sequences $\{\theta_n\}_{n\in\N}$ and $\{t^{\frac{s}{2}+\frac{1}{2}-\frac{3}{2p}}\theta_n\}_{n\in\N}$ are uniformly bounded respectively in $L^\infty((0,T);L^3 )$ and $L^\infty((0,T);\dot{W}^{s,p} )$.

We first show that $\{t^{\frac{s}{2}+\frac{1}{2}-\frac{3}{2p}}\theta_n\}_{n\in\N}$ is uniformly bounded in $L^\infty((0,T);\dot{W}^{s,p} )$. Using \eqref{4.1}, there exists a constant $C_1>0$ such that
\begin{align*}
\|\theta_1(\cdot,t)\|_{\dot{W}^{s,p}}\le\|G(t)\theta_0\|_{\dot{W}^{s,p}}\le C_1 t^{-\frac{s}{2}-\frac{1}{2}+\frac{3}{2p}}\|\theta_0\|_{L^3}.
\end{align*}
For any $n\ge1$, using the definition of $\theta_n$, we have
\begin{align}\label{4.6}
\sup_{0<t<T}t^{\frac{s}{2}+\frac{1}{2}-\frac{3}{2p}}\|\theta_{n+1}(\cdot,t)\|_{\dot{W}^{s,p}}\le C_1\|\theta_0\|_{L^3}+\|B(\theta_n(\cdot,t),\theta_n(\cdot,t))\|_{\dot{W}^{s,p}}.
\end{align}
Using \eqref{2.6}, \eqref{4.1}, \eqref{4.2-1} and \eqref{4.5}, there exists a constant $\tilde{C_1}>0$ such that
\begin{align*}
&\|B(\theta_n(x,t),\theta_n(x,t))\|_{\dot{W}^{s,p}}\\
&\le\tilde{C_1}\int_0^t(t-\tau)^{-\frac{1}{2}(s+1)}\|\nabla\cdot(u(\theta_n(\cdot,\tau))\theta_n(\cdot,\tau))\|_{\dot{W}^{-1,p}}d\tau\\
&\le\tilde{C_1}\int_0^t(t-\tau)^{-\frac{1}{2}(s+1)}\|u(\theta_n(\cdot,\tau))\theta_n(\cdot,\tau)\|_{L^p}d\tau\\
&\le\tilde{C_1}\int_0^t(t-\tau)^{-\frac{1}{2}(s+1)}\|u(\theta_n(\cdot,\tau))\|_{L^\infty}\|\theta_n(\cdot,\tau)\|_{L^p}d\tau\\
&\le\tilde{C_1}\int_0^t(t-\tau)^{-\frac{1}{2}(s+1)}\|u(\theta_n(\cdot,\tau))\|_{W^{2,p}}\|\theta_n(\cdot,\tau)\|_{L^p}d\tau\\
&\le\tilde{C_1}\int_0^t(t-\tau)^{-\frac{1}{2}(s+1)}\|\theta_n(\cdot,\tau)\|_{L^p}\|\theta_n(\cdot,\tau)\|_{L^p}d\tau\\
&\le\tilde{C_1}\left(\int_0^t(t-\tau)^{-\frac{1}{2}(s+1)}\tau^{-1+\frac{3}{p}}d\tau\right)\left(\sup_{0<t<T}\tau^{\frac{1}{2}-\frac{3}{2p}}\|\theta_{n}(\cdot,\tau)\|_{L^p}\right)^2\\
&\le4\delta^2\tilde{C_1} t^{-\frac{1}{2}s-\frac{1}{2}+\frac{3}{2p}}t^{\frac{3}{2p}}\left(\int_0^1(1-z)^{-\frac{1}{2}(s+1)}z^{-1+\frac{3}{p}}dz\right).
\end{align*}
By the assumptions on $s$ and $p$, we have $-\frac{1}{2}(s+1)>-1$, $-1+\frac{3}{p}>-1$ and $\frac{3}{2p}>0$. Hence
\begin{align*}
\mbox{$t^{\frac{3}{2p}}\le T^{\frac{3}{2p}}\le1$ and $\displaystyle\int_0^1(1-z)^{-\frac{1}{2}(s+1)}z^{-1+\frac{3}{p}}dz<\infty$.}
\end{align*}
Replacing $\tilde{C_1}$ if necessary, choosing $\delta<\frac{1}{2\sqrt{\tilde{C_1}}}$ and reducing $T$, we conclude from \eqref{4.6} that
\begin{align*}
\sup_{0<t<T}t^{\frac{1}{2}(s-s)+\frac{1}{2}-\frac{3}{2p}}\|\theta_{n+1}(\cdot,t)\|_{\dot{W}^{s,p}}\le C_1\|\theta_0\|_{\dot{W}^{s,3}}+1.
\end{align*}
Next we prove that $\{\theta_n\}_{n\in\N}$ is uniformly bounded in $L^\infty((0,T);L^3 )$. Again using \eqref{4.1}, there exists a constant $C_2>0$ such that
\begin{align*}
\|\theta_1(\cdot,t)\|_{L^3}\le\|G(t)\theta_0\|_{L^3}\le C_2\|\theta_0\|_{L^3},
\end{align*}
and by the definition of $\theta_n$, for $n\ge1$,
\begin{align*}
\sup_{0<t<T}\|\theta_{n+1}(\cdot,t)\|_{L^3}\le C_2\|\theta_0\|_{L^3}+\|B(\theta_n(\cdot,t),\theta_n(\cdot,t)\|_{L^3}.
\end{align*}
Using \eqref{2.6}, \eqref{4.1}, \eqref{4.2-1} and \eqref{4.5}, there exists a constant $\tilde{C_2}>0$ such that the term $\|B(\theta_n(\cdot,t),\theta_n(\cdot,t)\|_{L^3}$ can be estimated as follows.
\begin{align*}
&\|B(\theta_n(\cdot,t),\theta_n(\cdot,t)\|_{L^3}\\
&\le\tilde{C_2}\int_0^t(t-\tau)^{-\frac{1}{2}}\|\nabla\cdot(u(\theta_n(\cdot,\tau))\theta_n(\cdot,\tau))\|_{\dot{W}^{-1,3}}d\tau\\
&\le\tilde{C_2}\int_0^t(t-\tau)^{-\frac{1}{2}}\|u(\theta_n(\cdot,\tau))\theta_n(\cdot,\tau)\|_{L^3}d\tau\\
&\le\tilde{C_2}\int_0^t(t-\tau)^{-\frac{1}{2}}\|u(\theta_n(\cdot,\tau))\|_{L^\infty}\|\theta_n(\cdot,\tau)\|_{L^3}d\tau\\
&\le\tilde{C_2}\int_0^t(t-\tau)^{-\frac{1}{2}}\|u(\theta_n(\cdot,\tau))\|_{W^{2,p}}\|\theta_n(\cdot,\tau)\|_{L^3}d\tau\\
&\le\tilde{C_2}\int_0^t(t-\tau)^{-\frac{1}{2}}\|\theta_n(\cdot,\tau)\|_{L^p}\|\theta_n(\cdot,\tau)\|_{L^3}d\tau\\
&\le2\varepsilon\tilde{C_2}\int_0^t(t-\tau)^{-\frac{1}{2}}\tau^{-\frac{1}{2}+\frac{3}{2p}}\|\theta_n(\cdot,\tau)\|_{L^3}d\tau\\
&\le2\varepsilon\tilde{C_2}T^\frac{3}{2p}\left(\int_0^1(1-z)^{-\frac{1}{2}}z^{-\frac{1}{2}+\frac{3}{2p}}dz\right)\left(\sup_{0<\tau<T}\|\theta_n(\cdot,\tau)\|_{L^3}\right).
\end{align*}
Since $T\le1$ and $-\frac{1}{2}+\frac{3}{2p}>-1$, replacing $\tilde{C_2}$ if necessary, we obtain
\begin{align}\label{4.7}
\sup_{0<t<T}\|\theta_{n+1}(\cdot,t)\|_{L^3}\le C_2\|\theta_0\|_{L^3}+2\delta\tilde{C_2}\sup_{0<\tau<T}\|\theta_n(\cdot,\tau)\|_{L^3}.
\end{align}
By choosing $\delta$ small enough and reducing $T$ if necessary, an induction argument on \eqref{4.7} shows that $\{\theta_n\}_{n\in\N}$ is uniformly bounded in $L^\infty((0,T);L^3 )$.

Since the sequence $\{\theta_n\}_{n\in\N}$ is uniformly bounded in $L^\infty((0,T);L^3 )$, we can see that there exists a subsequence of $\{\theta_n\}_{n\in\N}$ which converges towards some $\tilde{\theta}$ weak-* in $L^\infty((0,T);L^3 )$ and consequently in $\mathcal{D}'(\R^3\times(0,T))$. On the other hand, we know that $\theta_n\rightarrow\theta^{\varepsilon_\kappa}$ in $\mathcal{E}_T$, which implies convergence in $\mathcal{D}'(\R^3\times(0,T))$ as well. Therefore,
\begin{align*}
\theta^{\varepsilon_\kappa}=\tilde\theta\in L^\infty((0,T);L^3 ).
\end{align*}
The time-continuity of $\theta^{\varepsilon_\kappa}$ follows by using the fact that $\theta^{\varepsilon_\kappa}$ belongs to $\mathcal{E}_T$ and it solves \eqref{4.2} (see Kato \cite{K92}-\cite{K94}). Hence we obtain \eqref{4.1-1}.

By similar methods, we can prove that $t^{\frac{s}{2}+\frac{1}{2}-\frac{3}{2p}}\theta^{\varepsilon_\kappa}\in C((0,T);\dot{W}^{s,p} )$ as well, and we finish the proof of Theorem~5.3.
\end{proof}
\begin{remark}
We point out that, the indexes $s$, $p$ as appeared in Theorem~5.3 are independent of each other. In other words, there is no further restriction on $s,p$ except $s\in[0,1)$ and $p\in(3,\infty)$.
\end{remark}

We give the proof of Theorem~2.3 which involves showing the global-in-time existence of $\theta^{\varepsilon_\kappa}$ and the time decay of $\|\theta^{\varepsilon_\kappa}(\cdot,t)\|_{\dot{W}^{s,p}}$ for all $s\in(0,1]$ and $p\in(3,\infty)$ when $t>0$.
\begin{proof}[{\bf proof of Theorem~2.3}]
In view of Theorem~5.3, we first extend the local solution $\theta^{\varepsilon_\kappa}$ satisfying \eqref{4.1-1}-\eqref{4.1-2}. Standard parabolic theory (see for example \cite{K92}) shows that
\begin{align}\label{4.10-1}
\partial^m_t \nabla^l_x \theta^{\varepsilon_\kappa}(x,t)\in C((0,T);L^3 \cap L^p )
\end{align}
for all $p>3$, $m\in\{0\}\cup\N$ and multi-index $l\in (\{0\}\cup\N)^3$, where $T>0$ is the existence time obtained in Theorem~5.3. Hence, we have $\theta^{\varepsilon_\kappa}\in C^\infty(\R^3\times(0,T))$ and $\theta(t)\in L^\infty $ for all $t\in(0,T)$.

Now we prove \eqref{1.4-1}. The following argument is reminiscent of the one given in \cite{FL14}. We notice that for $\theta_0\in L^3 \cap L^p $, the existence time $T$ (as in Theorem~5.3) can be chosen as
\begin{align*}
T=\left(\frac{\delta}{C_0\|\theta_0\|_{L^p}}\right)^{\frac{2p}{p-3}},
\end{align*}
where $0<\delta<\frac{1}{4K}$, and $K,C_0>0$ are given as in the proof of Theorem~5.3.

Let $\theta_0\in L^3 $. By Theorem~5.3, there exists constants $d_1,T_0>0$ and a unique solution $\theta^{\varepsilon_\kappa}$ to \eqref{1.1}-\eqref{1.1-13} defined on $(0,T_0)$ such that
\begin{align*}
\mbox{$\displaystyle\sup_{0<t<T_0}t^{\frac{1}{2}-\frac{3}{2p}}\|\theta^{\varepsilon_\kappa}(\cdot,t)\|_{L^p}\le d_1$, and $\displaystyle\sup_{0<t<T_0}\|\theta^{\varepsilon_\kappa}(\cdot,t)\|_{L^3}\le\|\theta_0\|_{L^3}$}.
\end{align*}
Next we define
\begin{align*}
T=\sup\left\{t>0:\theta^{\varepsilon_\kappa}\in C((0,t);L^3\cap L^p), \sup_{0<s<t}t^{\frac{1}{2}-\frac{3}{2p}}\|\theta^{\varepsilon_\kappa}(\cdot,s)\|_{L^p}<\infty, \sup_{0<s<t}\|\theta^{\varepsilon_\kappa}(\cdot,s)\|_{L^3}\le\|\theta_0\|_{L^3}.\right\}
\end{align*}
We claim that $T=\infty$. We argue by contradiction. Suppose that $T<\infty$, and we let $\delta\in(0,\frac{T}{2})$ which will be chosen later. By Theorem~5.3, we have $\theta^{\varepsilon_\kappa}(T-\delta)\in L^3\cap L^p$ and $\|\theta^{\varepsilon_\kappa}(T-\delta)\|_{L^p}\le\|\theta^{\varepsilon_\kappa}(\frac{T}{2})\|_{L^p}$. By choosing $\theta^{\varepsilon_\kappa}(T-\delta)$ as initial data, given $d_2\in(0,\frac{1}{4d_1})$, there exists $T_1>0$ and a unique solution $\tilde\theta^{\varepsilon_\kappa}$ to \eqref{1.1}-\eqref{1.1-13} defined on $(T-\delta,T-\delta+T_1)$ such that
\begin{align*}
\tilde\theta^{\varepsilon_\kappa}\in C((T-\delta,T-\delta+T_1);L^3\cap L^p),
\end{align*}
\begin{align*}
\sup_{T-\delta<t<T-\delta+T_1}[t-(T-\delta)]^{\frac{1}{2}-\frac{3}{2p}}\|\tilde\theta^{\varepsilon_\kappa}(t)\|_{L^p}&\le 2d_1,\\
\sup_{T-\delta<t<T-\delta+T_1}\|\tilde\theta^{\varepsilon_\kappa}(t)\|_{L^p}&\le\|\theta^{\varepsilon_\kappa}(T-\delta)\|_{L^3}.
\end{align*}
By the uniqueness of solution, we have $\theta^{\varepsilon_\kappa}(t)=\tilde\theta^{\varepsilon_\kappa}(t)$ for all $t\in(T-\delta,T)$. We now choose
\begin{align*}
\mbox{$\displaystyle T_1=\min\left\{\left(\frac{d_1}{C_0\|\theta(\frac{T}{2})\|_{L^p}}\right)^{\frac{2p}{p-3}},T\right\}$, $\delta\in(0,\min\{\frac{T}{2},T_1\})$ and $T_2=T-\delta+T_1$},
\end{align*}
then $T<T_2$ and there exists a solution $\bar\theta^{\varepsilon_\kappa}\in C((0,T_2);L^3\cap L^p)$ to \eqref{1.1}-\eqref{1.1-13} such that
\begin{align*}
\sup_{0<s<\tilde T}t^{\frac{1}{2}-\frac{3}{2p}}\|\theta^{\varepsilon_\kappa}(\cdot,s)\|_{L^p}<\infty, \sup_{0<s<\tilde T}\|\theta^{\varepsilon_\kappa}(\cdot,s)\|_{L^3}\le\|\theta_0\|_{L^3}
\end{align*}
for all $\tilde T\in(0,T_2)$, which contradicts the maximality of $T$. Hence we must have $T=\infty$ and we finish the proof of \eqref{1.4-1}.

Finally we consider \eqref{1.4-2}. We first claim that there exists a constant $C_3>0$ independent of $t$ such that
\begin{align}\label{4.10-2}
\|\theta^{\varepsilon_\kappa}(\cdot,t)\|_{L^\infty}\le C_3 t^{-\frac{1}{2}}
\end{align}
for all $t>0$. The proof for \eqref{4.10-2} is the same as the one given in \cite{FL14} and we include here for completeness. Because $\theta^{\varepsilon_\kappa}$ exists for all time, it follows from \eqref{2.6} and \eqref{4.10-1} that
\begin{align*}
\mbox{$\|\theta^{\varepsilon_\kappa}(\cdot,t)\|_{L^\infty}<\infty$ for all $t>0$}.
\end{align*}
In particular, $\theta^{\varepsilon_\kappa}(t)$ satisfies \eqref{1.1} in the classically sense for all $t>0$. In view of $\nabla\cdot(u(\theta^{\varepsilon_\kappa}))=0$, we integrate \eqref{1.1} by parts to obtain, for any $q\ge1$,
\begin{align}\label{4.8}
\frac{\partial}{\partial t}\|\theta^{\varepsilon_\kappa}(\cdot,t)\|_{L^q}^q&=q\int \theta^{\varepsilon_\kappa}(x,t)^{q-1}\frac{\partial}{\partial t}\theta^{\varepsilon_\kappa}(x,t)dx\notag\\
&=q\int \theta^{\varepsilon_\kappa}(x,t)^{q-1}(\varepsilon_\kappa\Delta\theta^{\varepsilon_\kappa}(x,t)-\nabla\cdot(u(\theta^{\varepsilon_\kappa})\theta^{\varepsilon_\kappa}(x,t)))dx\notag\\
&=q\int \theta^{\varepsilon_\kappa}(x,t)^{q-1}(\varepsilon_\kappa\Delta\theta^{\varepsilon_\kappa}(x,t))dx\notag\\
&\le-\varepsilon_\kappa\int |\nabla(\theta^{\varepsilon_\kappa}(x,t)^\frac{q}{2})|^2dx.
\end{align}
for all $t>0$. Using the Gagliardo-Nirenberg inequality for $\R^3$, there exists a constant $C_4>0$ such that
\begin{align*}
\|\theta^{\varepsilon_\kappa}(\cdot,t)\|^\frac{5q}{3}_{L^q}\le C_4\|\theta^{\varepsilon_\kappa}(\cdot,t)\|_{L^\frac{q}{2}}^\frac{2q}{3}\|\nabla(\theta^{\varepsilon_\kappa}(\cdot,t)^\frac{q}{2})\|_{L^2}^2,
\end{align*}
and hence we obtain the following inequality
\begin{align}\label{4.9}
\frac{\partial}{\partial t}\|\theta^{\varepsilon_\kappa}(\cdot,t)\|_{L^q}^q\le -C_4 (\|\theta^{\varepsilon_\kappa}(\cdot,t)\|_{L^\frac{q}{2}}^\frac{q}{2})^{-\frac{4}{3}}(\|\theta^{\varepsilon_\kappa}(\cdot,t)\|_{L^q}^q)^{\frac{5}{3}}.
\end{align}
By considering the sequence $q_k=3\cdot 2^k$ for $k\ge0$, we can solve \eqref{4.9} inductively to get
\begin{align*}
\|\theta^{\varepsilon_\kappa}(\cdot,t)\|_{L^{q_k}}^{q_k}\le A_{q_k}t^{-\frac{3}{2}(2^k-1)},
\end{align*}
where $A_{q_k}$ are defined by
\begin{align*}
\mbox{$A_{q_0}=\|\theta_0\|^3_{L^3}, \quad A_{q_k}^\frac{1}{q_k}=\left(\frac{3(2^k-1)}{2C_4}\right)^\frac{1}{2^{k+1}}A_{q_{k-1}}^\frac{1}{q_{k-1}}$ \,for $k\in\N$,}
\end{align*}
and we have used \eqref{4.8} to get $\|\theta^{\varepsilon_\kappa}(\cdot,t)\|^3_{L^3}\le A_{q_0}$ for all $t>0$. Therefore,
\begin{align*}
\|\theta^{\varepsilon_\kappa}(\cdot,t)\|_{L^{q_k}}\le \left(\prod_{i=1}^k\left(\frac{3(2^i-1)}{2C_4}\right)^\frac{1}{2^{i+1}}\right)A_{q_0}^\frac{1}{q_0} t^{-\frac{(2^k-1)}{2^{k+1}}}.
\end{align*}
Taking $k\rightarrow\infty$, we obtain a constant $C_5>0$ such that, for $t>0$,
\begin{align*}
\|\theta^{\varepsilon_\kappa}(\cdot,t)\|_{L^\infty}\le C_5\|\theta_0\|_{L^3}t^{-\frac{1}{2}},
\end{align*}
which implies \eqref{4.10-2} by choosing $C_3=C_5\|\theta_0\|_{L^3}$.

By interpolation, for $p>3$, we can further obtain a constant $C_6>0$ such that
\begin{align}\label{4.11}
\|\theta^{\varepsilon_\kappa}(\cdot,t)\|_{L^p}\le C_6 t^{-\frac{1}{2}+\frac{3}{2p}}.
\end{align}
To estimate $\|\theta^{\varepsilon_\kappa}(\cdot,t)\|_{W^{s,p}}$, using \eqref{4.1}, \eqref{4.2-1}, \eqref{4.10-2}, \eqref{4.11} and the integral equation \eqref{4.2}, there exists constants $C_7,C_8>0$ such that for all $t>0$,
\begin{align*}
\|\theta^{\varepsilon_\kappa}(\cdot,t)\|_{\dot{W}^{s,p}}&\le\|G(t)\theta_0\|_{\dot{W}^{s,p}}+\|B(\theta^{\varepsilon_\kappa},\theta^{\varepsilon_\kappa})(\cdot,t)\|_{\dot{W}^{s,p}}\\
&\le C_1 t^{-\frac{s}{2}-\frac{1}{2}+\frac{3}{2p}}\|\theta_0\|_{L^3}+C_7\int_0^t(t-\tau)^{-\frac{(s+1)}{2}}\|\nabla\cdot(u(\theta^{\varepsilon_\kappa})\theta^{\varepsilon_\kappa})\|_{\dot{W}^{-1,p}}d\tau\\
&\le C_1 t^{-\frac{s}{2}-\frac{1}{2}+\frac{3}{2p}}\|\theta_0\|_{L^3}+C_7\int_0^t(t-\tau)^{-\frac{(s+1)}{2}}\|u(\theta^{\varepsilon_\kappa})\theta^{\varepsilon_\kappa}\|_{L^p}d\tau\\
&\le C_1 t^{-\frac{s}{2}-\frac{1}{2}+\frac{3}{2p}}\|\theta_0\|_{L^3}+C_7\int_0^t(t-\tau)^{-\frac{(s+1)}{2}}\|u(\theta^{\varepsilon_\kappa})\|_{L^p}\|\theta^{\varepsilon_\kappa}\|_{L^\infty}d\tau\\
&\le C_1 t^{-\frac{s}{2}-\frac{1}{2}+\frac{3}{2p}}\|\theta_0\|_{L^3}\\
&\qquad+C_7\left(\int_0^t(t-\tau)^{-\frac{(s+1)}{2}}\tau^{-1+\frac{3}{2p}}d\tau\right)\left(\sup_{\tau>0}\tau^{\frac{1}{2}-\frac{3}{2p}}\|\theta^{\varepsilon_\kappa}\|_{L^p}\right)\left(\sup_{\tau>0}\tau^\frac{1}{2}\|\theta^{\varepsilon_\kappa}\|_{L^\infty}\right)\\
&\le C_1 t^{-\frac{s}{2}-\frac{1}{2}+\frac{3}{2p}}\|\theta_0\|_{L^3}+C_3C_6C_7t^{-\frac{s}{2}-\frac{1}{2}+\frac{3}{2p}}\left(\int_0^1(1-z)^{-\frac{(s+1)}{2}}z^{-1+\frac{3}{2p}}dz\right)\\
&\le C_8 t^{-\frac{s}{2}-\frac{1}{2}+\frac{3}{2p}},
\end{align*}
hence \eqref{1.4-2} follows and we have $\|\theta^{\varepsilon_\kappa}(\cdot,t)\|_{\dot{W}^{s,p}}\rightarrow0$ as $t\rightarrow\infty$. We finish the proof of Theorem~2.3.
\end{proof}

\section{Convergence of solutions when $\varepsilon_\kappa\rightarrow0$}

Let $\theta^{\varepsilon_\kappa}$ be the solution to \eqref{1.1}-\eqref{1.1-13} when $\varepsilon_\kappa>0$ with initial data $\theta_0\in L^3 $ as obtained in Theorem~2.3. We now consider the convergence of $\theta^{\varepsilon_\kappa}$ as $\varepsilon_\kappa\rightarrow0$. We first prove the following lemma which gives some uniform bounds on $\theta^{\varepsilon_\kappa}$ independent of time and $\varepsilon_\kappa$.

\begin{lemma}
Given $\theta_0\in L^3 $, there exist constants $c_1,c_2>0$ independent of $t$ and $\varepsilon_\kappa$ such that, for any $t\ge0$ and $\varepsilon_\kappa>0$,
\begin{align}\label{5.1}
\|\theta^{\varepsilon_\kappa}(\cdot,t)\|_{L^3}\le c_1\|\theta_0\|_{L^3},
\end{align}
\begin{align}\label{5.2}
\|u(\theta^{\varepsilon_\kappa})(\cdot,t)\|_{L^\infty}\le c_2\|\theta(\cdot,t)\|_{L^3}\le c_2\|\theta_0\|_{L^3}.
\end{align}
Furthermore, for any $q\ge1$, there exists $c_3(q)>0$ independent of $t$ and $\varepsilon_\kappa$ such that
\begin{align}\label{5.3}
\|u(\theta^{\varepsilon_\kappa})(\cdot,t)\|_{W^{2,q}}\le c_3(q)\|\theta(\cdot,t)\|_{L^q}\le c_3(q)\|\theta_0\|_{L^q}.
\end{align}
\end{lemma}
\begin{proof}
To show \eqref{5.1}, We choose $q=3$ in \eqref{4.8} to obtain
\begin{align*}
\frac{\partial}{\partial t}\|\theta^{\varepsilon_\kappa}(\cdot,t)\|_{L^3}^3&=2\int \theta^{\varepsilon_\kappa}(x,t)^2\frac{\partial}{\partial t}\theta^{\varepsilon_\kappa}(x,t)dx\le-\varepsilon_\kappa\int |\theta^{\varepsilon_\kappa}\nabla\theta^{\varepsilon_\kappa}(x,t)|^2dx,
\end{align*}
and \eqref{5.1} follows. \eqref{5.2}-\eqref{5.3} can be proved by similar arguments given in the proof of Lemma~4.1-4.2, and we omit the details here.
\end{proof}
The convergence of $\{\theta^{\varepsilon_\kappa}\}_{\varepsilon_\kappa>0}$ can be proved by a standard argument involving weak limits (for example, see Hoff \cite{H97}), while the convergence of $\displaystyle\varepsilon_\kappa\int_0^T\int |\nabla\theta^{\varepsilon_\kappa}|^2$ can be shown by controlling the term $\|\nabla u(\theta^{\varepsilon_\kappa})(\cdot,t)\|_{L^\infty}$ when $\nabla\theta_0\in L^2 $.
\begin{proof}[{\bf proof of Theorem~2.5}]
Without loss of generality, we assume that $\varepsilon_\kappa\le1$ and $\|\theta_0\|_{L^3}\le1$. Let $\{t_j\}$ be a countable dense subset of $[0,\infty)$. By \eqref{5.1}, we have that $\{\theta^{\varepsilon_\kappa}\}_{\varepsilon_\kappa>0}$ is bounded in $L^3 $, uniformly in $\varepsilon_\kappa$. Using a diagonalization process, we can obtain a sequence $\varepsilon_{\kappa_n}$ with $\displaystyle\lim_{n\rightarrow\infty}\varepsilon_{\kappa_n}=0$ for which $\theta^{\varepsilon_{\kappa_j}}$ converges weakly in $L^3 $, say to $\tilde\theta(\cdot,t_j)$ for each $t_j$.

Given $\phi\in W^{2,\frac{3}{2}}$, using \eqref{5.1}-\eqref{5.3}, we have for all $\varepsilon_\kappa>0$,
\begin{align*}
&\left|\int \theta^{\varepsilon_\kappa}(x,\cdot)dx\Big|_{t_1}^{t_2}\right|\\
&\le\left|\int_{t_1}^{t_2}\int \varepsilon_\kappa(\Delta\theta^{\varepsilon_\kappa})\phi dxdt-\int_{t_1}^{t_2}\int \phi u(\theta^{\varepsilon_\kappa})\cdot\nabla\theta^{\varepsilon_\kappa} dxdt\right|\\
&\le\varepsilon_\kappa\int_{t_1}^{t_2}\int |\theta^{\varepsilon_\kappa}||\Delta\phi|dxdt+\int_{t_1}^{t_2}\|u(\theta^{\varepsilon_\kappa}(\cdot,t))\|_{L^\infty}\|\theta^{\varepsilon_\kappa}(\cdot,t)\|_{L^3}\|\nabla\phi(\cdot,t)\|_{L^\frac{3}{2}}dt\\
&\le c_4|t_1-t_2|\|\phi\|_{W^{2,\frac{3}{2}}}\|\theta_0\|_{L^3},
\end{align*}
for some constant $c_4>0$ independent of time and $\varepsilon_\kappa$. So it follows that $\{\theta^{\varepsilon_{\kappa_n}}(\cdot,t)\}_{n\in\N}$ converges weakly to an element $\tilde\theta(\cdot,t)\in W^{-2,3} $ for every $t\in[0,\infty)$. On the other hand, using the uniform bound \eqref{5.1}, we have
\begin{align*}
\mbox{$\|\theta^{\varepsilon_{\kappa_n}}(\cdot,t)\|_{L^3}\le c_1\|\theta_0\|_{L^3}$ for all $\varepsilon_{\kappa_n}$.}
\end{align*}
It shows that every subsequence of $\theta^{\varepsilon_{\kappa_n}}$ has a further subsequence (still call it $\theta^{\varepsilon_{\kappa_n}}$ for simplicity) which converges weakly in $L^3 $, necessarily to $\tilde\theta(\cdot,t)$, for every $t\in[0,\infty)$. Hence $\tilde\theta(\cdot,t)\in L^3 $. The assertion that $\tilde\theta$ solves \eqref{1.1}-\eqref{1.1-13} with $\varepsilon_\kappa=0$ in the weak sense can be verified in a very similar way as given by Friedlander and Vicol in \cite{FV11a} for the {\it inviscid} MG equation (i.e. $\varepsilon_\nu=0$) (see Appendix A in \cite{FV11a}). By the uniqueness of the solution $\theta$ to \eqref{1.1}-\eqref{1.1-13} with $\varepsilon_\kappa=0$ as obtained in Theorem~2.1, we conclude $\tilde\theta=\theta$ and therefore $\theta^{\varepsilon_{\kappa_n}}(\cdot,t)\rightarrow\theta(\cdot,t)$ weakly in $L^3 $ for every $t\in[0,\infty)$.

Finally, it remains to prove \eqref{1.7}. We assume further that $\nabla\theta_0\in L^2 $.
in view of \eqref{4.10-1}, we can differentiate $\partial_t \theta^{\varepsilon_\kappa}+u(\theta^{\varepsilon_\kappa})\cdot\nabla\theta^{\varepsilon_\kappa}=\varepsilon_\kappa\Delta\theta^{\varepsilon_\kappa}$ with respect to $x$ and integrate to obtain
\begin{align*}
\frac{\partial}{\partial t}\int |\nabla\theta^{\varepsilon_\kappa}(t)|^2dx+\frac{\varepsilon_\kappa}{2}\int |\Delta\theta^{\varepsilon_\kappa}(t)|^2\le\int \|\nabla u(\theta^{\varepsilon_\kappa})(t)\|_{L^\infty}|\nabla\theta^{\varepsilon_\kappa}(t)|^2dx.
\end{align*}
for all $t\in(0,T)$. Using Gronwall's inequality, there exists a constant $c_5>0$ independent of $t,\varepsilon_\kappa$ such that
\begin{align*}
\|\nabla\theta^{\varepsilon_\kappa}(\cdot,t)\|_{L^2}&\le c_5\|\nabla\theta_0\|_{L^2}\exp\left(\int_0^t\|\nabla u(\theta^{\varepsilon_\kappa})(\cdot,\tau)\|_{L^\infty}d\tau\right)\\
&\le c_5\|\nabla\theta_0\|_{L^2}\exp\left(\int_0^t\|\nabla u(\theta^{\varepsilon_\kappa})(\cdot,\tau)\|_{L^\infty}d\tau\right).
\end{align*}
It suffices to estimate $\|\nabla u(\theta^{\varepsilon_\kappa})(\cdot,\tau)\|_{L^\infty}$. Using \eqref{2.3-1},  \eqref{2.6} and \eqref{5.3}, there exists a constant $c_6>0$ independent of $\tau,\varepsilon_\kappa$ such that for all $\tau\in(0,t)$,
\begin{align*}
\|\nabla u(\theta^{\varepsilon_\kappa})(\cdot,\tau)\|_{L^\infty}&\le c_6\left(\|D_x u(\theta^{\varepsilon_\kappa})(\cdot,\tau)\|_{L^6}+\|D^2_x u(\theta^{\varepsilon_\kappa})(\cdot,\tau)\|_{L^6}\right)\\
&\le c_6\|\theta_0 \|_{L^6}\\
&\le c_6\|\nabla\theta_0 \|_{L^2}.
\end{align*}
Therefore we have
\begin{align}\label{5.4}
\|\nabla\theta^{\varepsilon_\kappa}(\cdot,t)\|_{L^2}\le c_5\|\nabla\theta_0\|_{L^2}\exp\left(t\,c_6\|\nabla\theta_0\|_{L^2}\right).
\end{align}
Integrating over $t$, we conclude from \eqref{5.4} that
\begin{align*}
\varepsilon_\kappa\int_0^T\int |\nabla\theta^{\varepsilon_\kappa}(x,t)|^2dxdt\le\varepsilon_\kappa\left(\frac{c_5^2\|\nabla\theta_0\|_{L^2}}{2c_6}\right)\left[\exp\left(2Tc_6\|\theta_0\|_{L^2}\right)-1\right]
\end{align*}
and \eqref{1.7} follows immediately by taking $\varepsilon_\kappa\rightarrow0$. We finish the proof of Theorem~2.5.
\end{proof}

\section{An example of unstable eigenvalues for the MG system}

In this section we demonstrate the existence of unstable eigenvalues for the forced MG equation
\begin{align}\label{7.1}
\partial_t \theta + u \cdot \nabla\theta = \varepsilon_\kappa \Delta \theta + S
\end{align}
where the divergence free velocity $u$ is obtained from $\theta$ via $u_j = (\widehat{M}_j \widehat{\theta})^\nu$ with the Fourier multiplier symbols given by (1.10)-(1.13). We consider perturbations of a particular steady state, namely
\begin{align}\label{7.2}
\Theta_0 = A\sin m x_3, \quad U_0 = 0, \quad S = \varepsilon_\kappa A m^2 \sin m x_3
\end{align}
where the amplitude $A$ is an arbitrary constant. We consider the linear evolution of the perturbation temperature $\theta (x, t)$ and velocity $u(x, t)$ and make the assumption that, for some fixed integers $k_1$ and $k_2, \theta (x, t)$ has the form
\begin{align}\label{7.3}
\theta (x, t) = e^{\sigma t} \sin (k_1 x_1) \sin (k_2 x_2) \sum\limits_{n \ge 1} c_n \sin (nx_3)
\end{align}
which from (1.10)-(1.13) implies that the perturbation velocity
\begin{align}\label{7.4}
u_3 (x, t) = e^{\sigma t} \sin (k_1 x_1) \sin (k_2x_2) \sum\limits_{n \ge 1} c_n \ds\frac{N^2 (k^2_1 + k^2_2) [k^2_2 + \varepsilon_\nu (k^2_1 + k^2_2 + n^2)^2]}{N^4n^2 (k^2_1 + k^2_2 + n^2) + [k^2_2 + \varepsilon_\nu (k^2_1 + k^2 + n^2)^2]^2} \sin (nx_3).
\end{align}
Substituting (7.2)-(7.4) into the linearized version of (7.1) gives
\begin{align}
& \sum\limits_{n \ge 1} (\sigma + \varepsilon_\kappa (k^2_1 + k^2_2 + n^2)) c_n \sin (nx_3) \notag\\
& + A m \cos (mx_3) \sum\limits_{n \ge 1} c_n \ds\frac{N^2 (k^2_1 + k^2_2) [k^2_2 + \varepsilon_\nu (k^2_1 + k^2_2 + n^2)^2]}{N^4n^2 (k^2_1 + k^2_2 + n^2) + [k^2_2 + \varepsilon_\nu (k^2_1 + k^2_2 + n^2)^2]^2} \sin (nx_3)= 0.
\label{7.5}\end{align}
We now follow the construction of an unstable eigenvalue $\sigma > 0$ and associated $C^\infty$ smooth eigenfunction satisfying (7.5) which is given in \cite{FV11b} for the inviscid MG equation. A suitable modification of the method of continued fractions utilized for the Navier-Stokes equations by Meshalkin and Sinai \cite{MS61} produces a characteristic equation which gives a positive {\it lower} bound on a real root $\sigma^*$.  This procedure applied to the viscous MG equation demonstrates the existence of an eigenvalue $\sigma^*$ such that
\begin{align}\label{7.6}
& \sigma^* \ge \ds\frac{AmN^2}{2} \ds\frac{(k^2_1+k^2_2)[k^2_2 + \varepsilon_\nu (k^2_1+k^2_2+m^2)^2]}{4N^4m^2(k^2_1+k^2_2+4m^2) + [k^2_2 + \varepsilon_\nu (k^2_1+k^2+4m^2)^2]^2} \notag\\
& - \varepsilon_\kappa (k^2_1+k^2_2+4m^2).
\end{align}
We note that the characteristic equation also produces an upper bound
\begin{align}\label{7.7}
& \sigma^* \le AmN^2 \ds\frac{(k^2_1+k^2_2)[k^2_2 + \varepsilon_\nu (k^2_1+k^2_2+4m^2)^2]}{N^4m^2(k^2_1+k^2_2+m^2) + [k^2_2 + \varepsilon_\nu (k^2_1+k^2_2+m^2)^2]^2} \notag\\
& - \varepsilon_\kappa (k^2_1+k^2_2+m^2).
\end{align}

From (7.6) it follows that we may choose the amplitude $A$ of the steady temperature distribution sufficiently large in terms of $\varepsilon_\kappa, m$ and $N^2$ so that there exists at least one pair of wave numbers $(k_1, k_2)$ that produce a positive eigenvalue $\sigma^*$.

We now discuss the implications of the lower bound (7.6) in terms of the limiting behavior of the small parameters $\varepsilon_\nu$ and $\varepsilon_\kappa$ and the existence eigenvalues whose magnitudes grow as $\varepsilon_\nu$ or $\varepsilon_\kappa$ tend to to zero.

\vskip 1pc

\noindent{\bf Case (i):} both $\varepsilon_\nu = 0$ and $\varepsilon_\kappa = 0$. In this case the estimate (7.6) gives
\begin{align}\label{7.8}
\sigma^* \ge \ds\frac{AmN^2}{2} \ds\frac{(k^2_1 + k^2_2) k^2_2}{4N^4m^2 (k^2_1 + k^2_2 + 4m^2) + k^4_2}.
\end{align}
Choosing $(k_1, k_2) \in \mathbb{Z}^2$ such that $k^2_2 = k_1 = j \ge m \ge 1$ we obtain that for any integer $j \ge m$ there exists an eigenvalue $\sigma^*$ bounded from below as
\begin{align*}
\sigma^* \ge j C
\end{align*}
and the upper bound (7.7) gives
\begin{align*}
\sigma^* \le 8j C
\end{align*}
where $C$ is a constant depending only on $A, m$ and $N^2$. Thus in the absence of viscous or thermal diffusion we can construct eigenvalues with arbitrarily large real part. This property was used in \cite{FV11b} to prove that the Cauchy problem for the inviscid, nondiffusive MG equation is ill-posed in the sense of Hadamard.

\vskip 1pc
\noindent{\bf Case (ii):} $\varepsilon_\nu = 0$ and $0 < \varepsilon_\kappa \ll 1$. In this case (7.6) gives
\begin{align}\label{7.9}
\sigma^* \ge \ds\frac{AmN^2}{2} \ds\frac{(k^2_1 + k^2_2) k^2_2}{4N^4m^2 (k^2_1 + k^2_2 + 4m^2) + k^4_2} - \varepsilon_\kappa (k^2_1 + k^2_2 + 4 m^2).
\end{align}
Hence there are at most finitely many values $(k_1, k_2)$ for which the lower bound is positive. The maximum value of this lower bound as a function of $k_1$ and $k_2$ occurs when
\begin{align}\label{7.10}
k_1 \approx 1/\varepsilon_\kappa \mbox{ and } k_2 \approx 1/\varepsilon^{1/2}_\kappa,
\end{align}
(7.7) and (7.9) give
\begin{align}\label{7.11}
\ds\frac{1}{\varepsilon_\kappa} (AC_1 - 1) \le \sigma^* \le \ds\frac{1}{\varepsilon_\kappa} (8AC_1 - 1)
\end{align}
where $C_1$ is a positive constant depending only on $m$ and $N^2$. Thus the lower bound (7.11) is positive for sufficiently large amplitudes $A$ and grows like $\varepsilon^{-1}_\kappa$ for $0 < \varepsilon_\kappa \ll 1$. Hence the limit of vanishing thermal diffusivity produces very large, but finite, unstable eigenvalues.

\vskip 1pc
\noindent{\bf Case (iii):} $\varepsilon_\kappa = 0$ and $0 < \varepsilon_\nu \ll 1$.  The lower bound (7.6) gives
\begin{align}\label{7.12}
\sigma^* \ge \ds\frac{AmN^2}{2} \ds\frac{(k^2_1 + k^2_2) [k^2_2 + \varepsilon_\nu (k^2_1 + k^2_2 + m^2)^2]}{4N^4m^2 (k^2_1 + k^2_2 + 4m^2) + [k^2_2 + \varepsilon_\nu (k^2_1+ k^2_2 + 4m^2)^2]^2}.
\end{align}
Hence in this case there are unstable eigenvalues for all values of $A$ and for all finite values of $k_1$ and $k_2$. The maximum values of this lower bound as a function of $k_1$ and $k_2$ occurs when
\begin{align}\label{7.13}
k_1 \approx (1/\varepsilon_\nu)^{1/3} \mbox{ and } k_2 \approx (1/\varepsilon_\nu)^{1/6},
\end{align}
(7.7) and (7.12) gives
\begin{align*}
\left(\ds\frac{1}{\varepsilon_\nu}\right)^{1/3} AC_2 \le \sigma^* \le \left(\ds\frac{1}{\varepsilon_\nu}\right)^{1/3} 8AC_2,
\end{align*}
where $C_2$ is a positive constant depending only on $m$ and $N^2$. Thus the limit of vanishing viscosity also produces very large, but finite, unstable eigenvalues.

\vskip 1pc
\noindent{\bf Case (iv):} $0 < \varepsilon_\kappa \ll 1$ and $0 < \varepsilon_\nu \ll 1$. As we observed the bound (7.6) implies that provided $A$ is sufficiently large there always exists unstable eigenvalues for some small values of $k_1$ and $k_2$.  To examine the behavior of the lower bound in terms of the limiting behavior  of the two small parameters we write $\varepsilon_\nu = \varepsilon^\alpha_\kappa$. The maximum value of the lower bound occurs when
\begin{align*}
k_1 \approx 1/\varepsilon_\kappa \mbox{ and } k_2 \approx 1/\varepsilon^{1/2}_\kappa
\end{align*}
where (7.6) gives
\begin{align}\label{7.14}
\sigma^* \ge \ds\frac{1}{\varepsilon_\kappa} \left[\ds\frac{A (C_2 + C_3 \varepsilon^{\alpha-3}_\kappa)}{C_4 + (C_2 + C_3 \varepsilon^{\alpha - 3}_\kappa)^2} - 1\right]
\end{align}
and (7.7) gives
\begin{align*}
\sigma^* \le \ds\frac{1}{\varepsilon_\kappa} \left[\ds\frac{8A (C_2 + C_3 \varepsilon^{\alpha-3}_\kappa)}{C_4 + (C_2 + C_3 \varepsilon^{\alpha - 3}_\kappa)^2} - 1\right]
\end{align*}
where $C_2, C_3$ and $C_4$ are positive constants depending only on $m$ and $N^2$. Thus when $\alpha \ge 3$ we again conclude that for sufficiently large $A$ there exist eigenvalues that grow like $\varepsilon^{-1}_k$ as $\varepsilon_\kappa$ approaches zero. In contrast, when $\alpha < 3$ the bound (7.14) is negative for $0 < \varepsilon_\kappa \ll 1$ and this construction does not produce large unstable eigenvalues, in other words when $\alpha < 3$ the viscosity prevents rapid instabilities. In particular, this observation implies that there will be no ``large" unstable eigenvalues in the context of Section 6 where $\varepsilon_\nu > 0$ and $\lim_{\varepsilon_\kappa} \to 0$.

We note that in each case (ii), (iii) and (iv) $(\alpha \ge 3)$ the maximum lower bound on $\sigma^*$ is attained on the parabolic curves $k_1 = k^2_2$ and this bound grows with $\varepsilon^{-1}_\nu$ or $\varepsilon^{-1}_\kappa$. The limit $\varepsilon_\nu \to 0$ and the limit $\varepsilon_\kappa \to 0$ are consistent with the inviscid, non-diffusive case (i) which permits {\it arbitrarily} large eigenvalues associated with the curves $k_1 = k^2_2$ as $k_1 \to \infty$.

We also note that when $\varepsilon_\kappa > 0$ instability for the linearized equation implies nonlinear Lyapunov instability for the active scalar equation (7.1). This result was proved for the inviscid MG equation in \cite{FV11b} using a modification of the bootstrap techinques applied to the Navier-Stokes equations in \cite{FPS06}. The proof in \cite{FV11b} directly carries over to the viscous MG equation where $\varepsilon_\nu > 0$.

\vskip 1pc

To summarize, the classical geodynamo problem of magnetic field generation via the flow of an electrically conducting fluid is closely connected with the existence of unstable eigenvalues (see, for example, \cite{M83}). We have shown that the Moffatt and Loper geodynamo model permits such instabilities. The active scalar equation (7.1) linearized about a suitable background temperature field supports unstable eigenvalues. If either one of $\varepsilon_\nu$ or $\varepsilon_\kappa$ is zero and the other is very small and positive there exist perturbations that grow very rapidly on certain parabolic curves in wave number space. Such fast instabilities also occur when both $\varepsilon_\nu$ and $\varepsilon_\kappa$ are positive with $\varepsilon_\nu < \varepsilon^3_\kappa$. The existence of very strong instabilities is a consequence of certain features of the structure of the Fourier multiplier symbol $\widehat{M}_3$ that are produced by the specific constitutive law in the model relating the temperature with the velocity and the magnetic field (see equations (1.5)-(1.7)). The crucial features are the anisotropy of the symbol and the fact that it is even with respect to the wave numbers $k_1, k_2$ and $k_3$.  In the case where both $\varepsilon_\kappa$ and $\varepsilon_\nu$ are zero a consequence of the structure of the symbol is the existence of {\it arbitrarily} large eigenvalues. This implies that the MG equation without the control of any diffusion is Hadamard ill-posed, in contrast with the diffusive MG equation which is globally well posed.


\newpage

\subsection*{Acknowledgment} We thank Vlad Vicol and the anonymous referees for very helpful advice and suggestions. S.F. is supported in part by NSF grant DMS-1207780.


\begin{thebibliography}{00}

\bibitem{AB11} J. Azzam and J. Bedrossian, {\it Bounded Mean Oscillation and the Uniqueness of Active Scalar Equations}, arXiv:1108.2735.

\bibitem{BK12} F. Bernicot and S. Keraani, {\it On the global well-posedness of the 2D Euler equation for a large class of Yudovich type data}, 	 arXiv:1204.6006.

\bibitem{BCD11} H. Bahouri, J. Chemin and R. Danchin, {\it Fourier Analysis and Nonlinear Partial Differential Equations}, Grundlehren der Mathematischen Wissenschaften 343 (Springer, 2011).

\bibitem{CF08} J. A. Carrillo and L.C.F. Ferreira, {\it The asymptotic behavior of subcritical dissipative quasi-geostrophic equations}, Nonlinearity 21 (5) (2008), 1001--1018.

\bibitem{DL89} R. J. DiPerna and P.-L. Lions, {\it Ordinary differential equations, transport theory and Sobolev space}, Invent. Math. 98 (1989), 511--547.

\bibitem{D14} S.A. Denisov, {\it Double exponential growth of the vorticity gradient for the two-dimensional Euler equation}, to appear in Proceedings of the AMS.

\bibitem{FL14-1} L.C.F. Ferreira and L.S.M. Lima, {\it Self-similar solutions for active scalar equations in Fourier-Besov-Morrey space}, Monatshefte fur Mathematik 175 (2014), 491--509.

\bibitem{FL14} L.C.F. Ferreira and L.S.M. Lima, {\it Global well-posedness and symmetries for dissipative active scalar equations with positive-order couplings}, arXiv:1305.2987.

\bibitem{FPS06} S. Friedlander, N. Pavlovi\'{c} and R. Shvydkoy, {\it Nonlinear instability for the Navier-Stokes equations}, Commun. Math. Phys., 264, 335-47.

\bibitem{FRV12} S. Friedlander, W. Rusin, and V. Vicol, {\it On the supercritically diffusive magneto-geostrophic equations}, Nonlinearity, 25(11):3071--3097, 2012.

\bibitem{FRV14} S. Friedlander, W. Rusin and V. Vicol, {\it The magneto-geostrophic equations: a survey}, Proceedings of the St. Petersburg Mathematical Society, Volume XV: Advances in Mathematical Analysis of Partial Differential Equations. (2014) D. Apushkinskaya, and A.I. Nazarov, eds. pp. 53--78.

\bibitem{FV11a} S. Friedlander and V. Vicol, {\it Global well-posedness for an advection-diffusion equation arising in magneto-geostrophic dynamics}, Ann. Inst. H. Poincar\'e Anal. Non Lin\'{e}aire, 28(2):283--301, 2011.

\bibitem{FV11b} S. Friedlander and V. Vicol, {\it On the ill/well-posedness and nonlinear instability of the magneto-geostrophic equations}, Nonlinearity, 24(11):3019--3042, 2011.

\bibitem{FV12} S. Friedlander and V. Vicol, {\it Higher regularity of H\"{o}lder continuous solutions of parabolic equations with singular drift velocities}, J. Math. Fluid Mech., 14(2):255--266, 2012.

\bibitem{H97} D. Hoff, {\it Discontinuous Solutions of the Navier-Stokes Equations for Multidimensional Flows of Heat-Conducting Fluids},  Arch. Ration. Mech. Anal., 139 (1997), pp. 303--354.

\bibitem{L72} J.E. Lewis, {\it The initial-boundary value problem for the Navier-Stokes equations with data in $L^p$}, Indiana Univ. Math. J. 22 (1972/73), 739--761.

\bibitem{MB01}A. J. Majda and A. Bertozzi, {\it Vorticity and incompressible flow}, Cambridge Texts in Applied
Mathematics (No. 27), Cambridge University Press (Cambridge), 2001.

\bibitem{MS61} L. Meshalkin and Y. Sinai, {\it Investigation of stability for a system of equations describing plane motion of a viscous incompressible fluid}, Appl. Math. Mech., 25, 1140-3, 1961.

\bibitem{ML94} H.K. Moffatt and D.E. Loper, {\it The magnetostrophic rise of a buoyant parcel in the earth's core}, Geophysical Journal
International, 117(2):394--402, 1994.

\bibitem{M83} H.K. Moffatt, {\it Magnetic Field Generation in Electrically Conducting Fluids}, Cambridge: Cambridge University Press, 1983.

\bibitem{M08} H.K. Moffatt, {\it Magnetostrophic turbulence and the geodynamo}, In Y. Kaneda, editor, IUTAM Symposium on Computational
Physics and New Perspectives in Turbulence, Nagoya, Japan, September, 11?14, 2006, volume 4 of IUTAM
Bookser., pages 339?346. Springer, Dordrecht, 2008.

\bibitem{NPV11} E. D. Nezzaa, G. Palatuccia and E. Valdinocia, {\it Hitchhiker's guide to the fractional Sobolev spaces}, Bull. Sci. math., Vol. 136 (2012), No. 5, 521--573.

\bibitem{S70} E.M. Stein, {\it Singular Integrals and Differentiability Properties of Function}, Princeton Univ. Press, 1970.

\bibitem{K92} T. Kato, {\it Strong solutions of the Navier-Stokes equation in Morrey Spaces}, Bol. Soc. Bras. Mat. 22 (2) (1992), 127--155.

\bibitem{K94} T. Kato, {\it The Navier-Stokes equation for an incompressible fluid in $\R^2$ with a measure as the initial vorticity}, Differential Integral Equation 7 (4-4) (1994), 949--966.

\bibitem{KP86} T. Kato and G. Ponce, {\it Well-Posedness of the Euler and Navier-Stokes Equations in the Lebesgue space $L^p_s(\R^2)$}, Revista Matem\'{a}tica Iberoamericana 2 (1-2) (1986), 73--88.

\bibitem{Z89} W. Ziemer, {\it Weakly differentiable functions}, Springer-Verlag, 1989.

\end{thebibliography}
\end{document}